\documentclass[11pt]{amsart}

\title[Shadows of acyclic 4-manifolds with sphere boundary]
{Shadows of acyclic 4-manifolds with sphere boundary}

\author{Yuya Koda}
\thanks{The first author is supported in part by JSPS KAKENHI Grant
 Numbers 15H03620, 17K05254, 17H06463, and JST CREST Grant Number JPMJCR17J4. 
The second author is supported by JSPS KAKENHI Grant Number 18H05827.}
\address{
Department of Mathematics \newline
\indent Hiroshima University, 1-3-1 Kagamiyama, Higashi-Hiroshima, 739-8526, Japan}
\email{ykoda@hiroshima-u.ac.jp}

\author{Hironobu Naoe}
\address{
Department of Mathematics
\newline
\indent Chuo University, 1-13-27 Kasuga Bunkyo-ku, Tokyo, 112-8551, Japan}
\email{naoe@math.chuo-u.ac.jp}

\usepackage{amsthm}
\usepackage{latexsym}
\usepackage[dvips]{graphicx}
\usepackage[dvips]{psfrag}
\usepackage[dvips]{color}
\usepackage{xypic}
\usepackage[abs]{overpic}
\usepackage{caption}
\usepackage[all]{xy}
\usepackage[normalem]{ulem}

\theoremstyle{plain}
\newtheorem*{theorem*}{Theorem}
\newtheorem*{lemma*} {Lemma}
\newtheorem*{corollary*} {Corollary}
\newtheorem*{proposition*}{Proposition}
\newtheorem*{conjecture*}{Conjecture}
\newtheorem{theorem}{Theorem}[section]
\newtheorem{lemma}[theorem]{Lemma}

\theoremstyle{remark}

\newtheorem*{definition}{Definition}
\newtheorem*{claim*}{Claim}

\newtheorem*{remark}{Remark}
\newtheorem*{question}{Question}

\theoremstyle{definition}

\newtheoremstyle{citing}
  {}
  {}
  {\itshape}
  {}
  {\bfseries}
  {.}
  {.5em}
  {\thmnote{#3}}

\theoremstyle{citing}

\newcommand{\Integer}{\mathbb{Z}}
\newcommand{\Real}{\mathbb{R}}
\newcommand{\Complex}{\mathbb{C}}

\newcommand{\Nbd}{\operatorname{Nbd}}

\newcommand{\Int}{\operatorname{Int}}
\newcommand{\gl}{\mathrm{gl}}



\begin{document}

\maketitle

\begin{abstract}
In terms of Turaev's shadows, we provide a sufficient condition for 
a compact, smooth, acyclic $4$-manifold with boundary the $3$-sphere 
to be diffeomorphic to the standard $4$-ball. 
As a consequence, we prove that if a compact, smooth, acyclic $4$-manifold with boundary the $3$-sphere 
has shadow-complexity at most $2$, then it is diffeomorphic to the standard $4$-ball. 
\end{abstract}

\vspace{1em}

\begin{small}
\hspace{2em}  \textbf{2010 Mathematics Subject Classification}: 
57N13; 57M20, 57R55, 57R65


\hspace{2em} 
\end{small}

\section*{Introduction}

In \cite{Tur92, Tur94}, Turaev introduced the notion of {\it shadow} 
as a combinatorial tool to present smooth 3- and 4-manifolds. 
A shadow of a 4-manifold $M$ with boundary is a simple polyhedron 
$X$ properly embedded in $M$ so that $M$ collapses onto $X$, 
and $X$ is locally flat in $M$. 
The polyhedron $X$ is also called a shadow of the 3-manifold $\partial M$. 
By counting the minimum number of vertices of a shadow of a given 4- or 3-manifold, 
we get a (non-negative) integer-valued invariant called 
the {\it shadow-complexity}. 

In the $3$-manifold topology, shadows are used to study quantum invariants, 
see e.g. \cite{Tur92, Tur94, Bur97, Shu97, Thu02, CM17}. 
Moreover, it was revealed that the shadow-complexity of a 3-manifold $M$ 
is strongly related to the Gromov norm and the minimum number of 
codimension-$2$ singular fibers of a stable map $M \to \Real^2$, see \cite{CT08, CFMP07, IK17}.  

In the dimension $4$, shadows allow us to classify $4$-manifolds 
experimentally according to increasing complexity. 
Costantino \cite{Cos06a} studied closed 4-manifolds of shadow-complexity $0$ or $1$ in
a special case. 
Here, a shadow of a closed 4-manifold is a shadow of the union of $0$, $1$, and $2$ handles of 
its handle decomposition. 
In \cite{Mar11} Martelli gave a complete classification of the closed 4-manifolds
of shadow-complexity $0$. 
A very interesting consequence of this paper is that 
a simply connected closed 4-manifold has complexity zero if and only
if it is a connected sum of copies of the standard 
$S^4$, $S^2 \times S^2$, $\Complex P^2$, and $\overline{\Complex P^2}$. 
This implies in particular that the shadow-complexity detects the exotic structures on those manifolds. 
It is also classified the closed 4-manifolds of shadow-complexity $1$ in 
\cite{KMN18}. 
For the other studies of 4-manifolds using shadows 
see e.g. \cite{Cos05a, Cos06b, Cos08, Cos05b, Mar05, Nao18, Nao, IN}. 

\vspace{1em}

In the present paper, we consider the following naive question. 
\begin{question}
Let $M$ be an acyclic $4$-dimensional $2$-handlebody with boundary the $3$-sphere. 
Then is $M$ diffeomorphic to the standard $4$-ball?
\end{question}
\noindent Here, recall that a compact, oriented 4-manifold is called a $2$-handlebody if 
it is made of finitely many handles of index at most $2$. 
Note that the manifold $M$ in the above question is at least \textit{homeomorphic} to the $4$-ball. 
Indeed, it is easy to see that $M$ is simply connected, thus, $M$ is homeomorphic to the $4$-ball by 
Freedman's classification theorem \cite{Fre82}. 
A negative answer to the above question 
implies the existence of an exotic $4$-sphere. 
The following theorem gives an affirmative answer to the question 
when the (special) shadow-complexity of $M$ is very small. 
\begin{theorem}
\label{thm:Costantino and Naoe}
\begin{enumerate}
\item
\label{thm:acyclic special polyhedron of complexity 0}
$(${\rm Costantino} \cite{Cos06a}$)$ 
Every acyclic $4$-manifold of special shadow-complexity $0$ or $1$ 
with boundary the $3$-sphere is diffeomorphic to the standard $4$-ball. 
Here, for the definition of the special shadow-complexity, see Section $\ref{sec:Shadows}$. 
\item
\label{thm:acyclic simple polyhedron of complexity 0}
$(${\rm Naoe} \cite{Nao17}$)$ 
Every acyclic $4$-manifold of shadow-complexity $0$ is 
diffeomorphic to the standard $4$-ball. 
\end{enumerate}
\end{theorem}

In this paper, using shadows we provide a sufficient condition for 
a compact, smooth, acyclic $4$-manifold with boundary the $3$-sphere 
to be diffeomorphic to the standard $4$-ball (Theorem \ref{thm:general setting}). 
As a direct consequence, we show 
(in Theorem \ref{thm:acyclic 4-manifold of connected shadow-complexity at most 2}) 
that every acyclic $4$-manifold $M$ of shadow-complexity at most $2$  
with $\partial M \cong S^3$ is diffeomorphic to the standard $4$-ball. 
Precisely speaking, we show the same thing for a wider class of 4-manifolds, 
that is, $4$-manifolds of \textit{connected shadow-complexity} at most $2$. 
See Section $\ref{sec:Shadows}$ for the definition.

Throughout the paper, we will work in the smooth category unless otherwise mentioned.

\section{Shadows}
\label{sec:Shadows}

A compact and connected polyhedron $X$ is called 
a {\it simple polyhedron} if 
every point of $X$ has a star neighborhood homeomorphic to 
one of the five models 
shown in Figure \ref{fig:simple_polyhedron}. 
\begin{figure}[htbp]
\begin{center}
\begin{minipage}{400pt}
\includegraphics[width=14cm,clip]{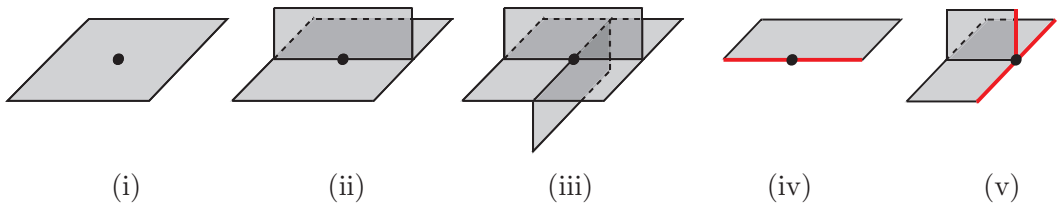}
\begin{picture}(400,0)(0,0)
\put(40,0){(i)}
\put(122,0){(ii)}
\put(205,0){(iii)}
\put(288,0){(iv)}
\put(370,0){(v)}
\end{picture}
\end{minipage}
\caption{The local models of a simple polyhedron.}
\label{fig:simple_polyhedron}
\end{center}
\end{figure}
A point whose star neighborhood is shaped on the model 
(iii) is called a {\it vertex} of $X$, and we denote the set of vertices of $X$  by $V(X)$. 
The set of points whose stat neighborhoods are shaped on the models (ii), (iii) or (v) is called the {\it singular set} of $X$, 
and we denote it by $S(X)$. 
The set of points whose star neighborhoods are shaped on the models  (iv) or (v) is called the {\it boundary} of $X$ and 
we denote it by $\partial X$. 
Each component of $X \setminus S(X)$ 
is called a \textit{region}, and we denote the set of regions of $X$  by $R(X)$. 
The number of vertices of $X$ is called the {\it complexity} of $X$. 
In \cite{KMN18}, the {\it connected complexity} of $X$ was defined to be the maximum number of vertices 
that are contained in some connected component of $S(X)$. 
A simple polyhedron $X$ is said to be {\it closed} if $\partial X = \emptyset$. 
A simple polyhedron $X$ is said to be {\it special} 
if each region of $X$ is simply-connected.  
We note that if $X$ is special and $X \not\cong D^2$, $X$ is closed and $S(X)$ is connected.

\begin{definition}
A simple polyhedron $X$ embedded in a compact oriented smooth 4-manifold $M$ is called a {\it shadow} of $M$ if 
\begin{itemize}
\item
$M$ collapses onto $X$ after equipping the natural PL structure on $M$; 
\item
$X$ is {\it locally flat}, that is, 
each point $x$ of $X$ has a neighborhood $\Nbd(x ; X)$ that lies in a 3-dimensional submanifold of $M$; and 
\item
$\partial M \cap X = \partial X$. 
\end{itemize}
\end{definition}
Note that $\partial X$ is a {\it knotted trivalent graph}, 
i.e. a smooth graph in $\partial M$ with only vertices of valence $3$, where 
we admits knot components as well.
For $k \in \{0,1,2,3\}$ a {\it $k$-handlebody} is defined to 
be an oriented $4$-manifold made of finitely many handles of index at most $k$.
In \cite{Tur92, Tur94}, Turaev proved that any 2-handlebody has a (special) shadow. 
In \cite{Cos05a, CT08}, the {\it shadow-complexity} ({\it special  shadow-complexity}) of 
a 2-handlebody $M$, 
denoted by $\mathrm{sc}(M)$ (resp. $\mathrm{sc}^{\mathrm{sp}}(M)$), 
was defined to be the minimum complexity of any shadow (resp. special shadow) of $M$. 
In \cite{KMN18}, the {\it connected shadow-complexity} of $M$, denoted by $\mathrm{sc^*}(M)$, 
was defined to be the minimum connected complexity of any shadow of $M$. 
Note that the shadow-complexity of $M$ is $0$ if and only if the connected shadow-complexity 
of $M$ is $0$. 
In general, we have $\mathrm{sc^*}(M) \leq \mathrm{sc}(M) \leq \mathrm{sc}^{\mathrm{sp}}(M)$. 

A {\it framed knotted trivalent graph} is a knotted trivalent graph 
equipped with a framing, i.e. an oriented surface
thickening of the graph considered up to isotopy. 
Let $M$ be a compact oriented smooth 4-manifold, and let $X \subset M$ be a shadow. 
Fix a framing of the knotted trivalent graph $\Gamma := \partial X$. 
To each region $R$ of $X$, we may assign a half-integer $\gl (R)$, called a {\it gleam}, 
as follows. 
Let $\iota : R \hookrightarrow M$ be the inclusion.  
Let $\bar{R}$ be the metric completion of $R$ with the path metric inherited from a Riemanian metric on $R$. 
Suppose for simplicity that the natural extension $\bar{\iota} : \bar{R} \to M$ is injective. 
The boundary $\partial \bar{R}$ of $\bar{R}$ consists of simple closed curves. 
The framing of $\Gamma$ and the germs $\Nbd(\bar{R}; X) \setminus R$ of the remaining regions near $\partial \bar{R}$ 
provide a structure of interval bundle over $\partial \bar{R}$, 
which is a sub-bundle of the normal bundle of $\partial \bar{R}$ in $M$. 
Let $\bar{R}'$ be a generic small perturbation of $\bar{R}$ such that $\partial \bar{R}'$ lies in the interval
bundle. 
The gleam $\gl (R)$ is then (well-)defined by counting the finitely many 
isolated intersections of $\bar{R}$ and $\bar{R}'$ with signs 
as follows: 
\[
\gl (R) = \frac{1}{2} \# (\partial \bar{R} \cap \partial \bar{R}') + \# (\Int\bar{R} \cap \Int\bar{R}') \in \frac{1}{2} \Integer . 
\]
We call a polyhedron $X$ equipped with a gleam on each region a {\it shadowed polyhedron}.  
In \cite{Tur92, Tur94}, Turaev showed that the 4-manifold $M$ and the framed knotted trivalent graph 
$\Gamma \subset \partial M$ are recovered from a shadowed polyhedron $(X, \gl)$ in a canonical way.

Let $K_i$ ($i=1,2$) be a framed oriented knot in the boundary of a compact oriented 4-manifold $M_i$. 
Let $B_i$ $(i=1,2)$ be a 3-ball in $\partial M_i$ such that 
$B_i \cap K_i$ is a properly embedded trivial arc in $B_i$. 
Let $g : (B_2, K_2) \to (B_1, K_1)$ be a diffeomorphism such that 
\begin{itemize}
\item
$g : B_2 \to B_1$ is orientation-reversing; 
\item
$g|_{K_2} : K_2 \to K_1$ is orientation-reversing; and 
\item
$g$ respects (the corresponding parts of) the framings. 
\end{itemize}
We denote the framed knot $K := (K_1 \setminus \Int D_1) \cup_{g|_{K_2 \cap \partial D_2}} (K_2 \setminus \Int D_2)$ 
in $\partial (M_1 \natural M_2) = \partial M_1 \# \partial M_2$ by 
$K_1 \# K_2$, and call it the {\it connected sum} of $K_1$ and $K_2$. 
The following two lemmas are straightforward from the definition. 

\begin{lemma}
\label{lem:connected sum of framed knots}
Let $K_i$ $(i=1,2)$ be a framed oriented knot in the boundary of a compact oriented $4$-manifold $M_i$. 
Let $f: \Nbd(K_2; \partial M_2) \to \Nbd(K_1; \partial M_1)$ be an orientation-reversing 
diffeomorphism such that 
$f|_{K_2} : K_2 \to K_1$ is orientation-reversing and 
$f$ respects the framings. 
Then the $4$-manifold $M_1 \cup_f M_2$ is obtained from $M_1 \natural M_2$ by attaching 
a $2$-handle along the framed knot $K_1 \# K_2$. 
\end{lemma}

\begin{lemma}
\label{lem:gluing formula for shadows}
Let $X_i$ $(i=1,2)$ be a shadowed polyhedron of a compact $4$-manifolds $M_i$. 
Let $K_i$ be a $($framed$)$ knot component of $\partial X_i$. 
Fix an orientation of each of $K_1$ and $K_2$. 
Let $f: \Nbd(K_2; \partial M_2) \to \Nbd(K_1; \partial M_1)$ be a diffeomorphism 
such that 
$f|_{K_2} : K_2 \to K_1$ is orientation-reversing and 
$f$ respects the framings. 
Then the shadowed polyhedron obtained by the move shown in Figure $\ref{fig:gluing_shadows}$ 
is a shadow of $M_1 \cup_f M_2$.
\end{lemma}
\begin{figure}[htbp]
\begin{center}
\begin{minipage}{10cm}
\includegraphics[width=10cm,clip]{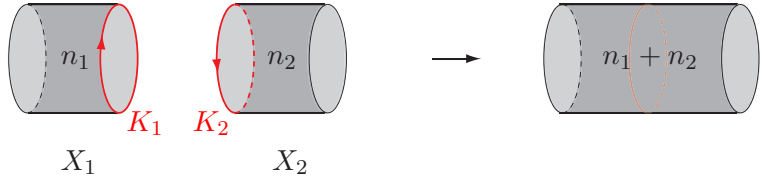}
\begin{picture}(400,0)(0,0)
\put(20,0){$X_1$}
\put(100,0){$X_2$}
\put(45,15){\color{red}$K_1$}
\put(70,15){\color{red}$K_2$}
\put(20,40){$n_1$}
\put(98,40){$n_2$}
\put(225,40){$n_1 + n_2$}
\end{picture}
\end{minipage}
\caption{This move gives a gluing formula for shadows.}
\label{fig:gluing_shadows}
\end{center}
\end{figure}

Let $X$ be a simple polyhedron. 
In general, a polyhedron obtained by collapsing $X$ might be 
no longer a simple polyhedron but an {\it almost-simple polyhedron}, 
i.e. a compact polyhedron where the link of each point 
can be embedded into the complete graph $\Gamma_4$ with $4$ vertices, 
see Matveev \cite{Mat03} for the details. 
A point of an almost-simple polyhedron is called a {\it true vertex} 
if its link is $\Gamma_4$, equivalently, 
the star neighborhood of the point is shaped on the model Figure \ref{fig:simple_polyhedron} (iii). 
An almost-simple polyhedron is said to be {\it minimal with respect to collapsing} 
if it cannot be collapsed onto any proper subpolyhedron. 
Up to a small perturbation, each point of such a polyhedron has a star neighborhood of 
one of Figure \ref{fig:simple_polyhedron} (i)-(iii) and 
Figure \ref{fig:almost-simple_polyhedron} (i)-(iv).  
\begin{figure}[htbp]
\begin{center}
\begin{minipage}{9cm}
\includegraphics[width=9cm,clip]{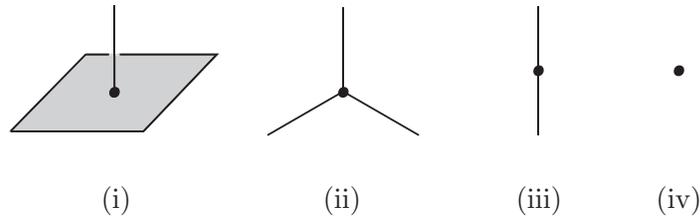}
\begin{picture}(400,0)(0,0)
\put(35,0){(i)}
\put(119,0){(ii)}
\put(192,0){(iii)}
\put(245,0){(iv)}
\end{picture}
\end{minipage}
\caption{Local models of an almost-simple polyhedron.}
\label{fig:almost-simple_polyhedron}
\end{center}
\end{figure}
Note that a simple polyhedron is minimal if and only if it is closed. 

\begin{lemma}
\label{lem:non-collapsible polyhedra}
Let $M$ be a $4$-manifold of shadow-complexity 
$($resp. connected shadow-complexity$)$ $n$ $(\geq 1)$. 
Then $M$ admits a closed shadow of complexity $($resp. connected complexity$)$ exactly $n$. 
\end{lemma}
\begin{proof}
Let $M$ be a $4$-manifold of shadow-complexity $n$ $(\geq 1)$. 
Let $X$ be a shadow of $M$ with exactly $n$ vertices. 
Then $X$ collapses onto an almost simple polyhedron $Y$ 
that is minimal with respect to collapsing and has at most $n$ true vertices. 
If $Y$ remains to be a simple polyhedron, there is nothing to prove. 
If $Y$ is a graph, i.e. a $0$- or $1$-dimensional complex, 
then $M$ is a $1$-handlebody, which contradicts the assumption that 
the shadow-complexity $n$ of $M$ is at least $1$. 
Suppose that $Y$ is not a graph. 
Then $Y$ is the union of a simple polyhedron $Y'$ and a graph $\Gamma$. 

If there exists a path, i.e. a subgraph homeomorphic to $[0,1]$,  
in $\Gamma$ that connects two different points of 
$Y'$, we apply the move shown in Figure \ref{fig:almost-simple_to_simple1}. 
\begin{figure}[htbp]
\begin{center}
\begin{minipage}{12cm}
\includegraphics[width=12cm,clip]{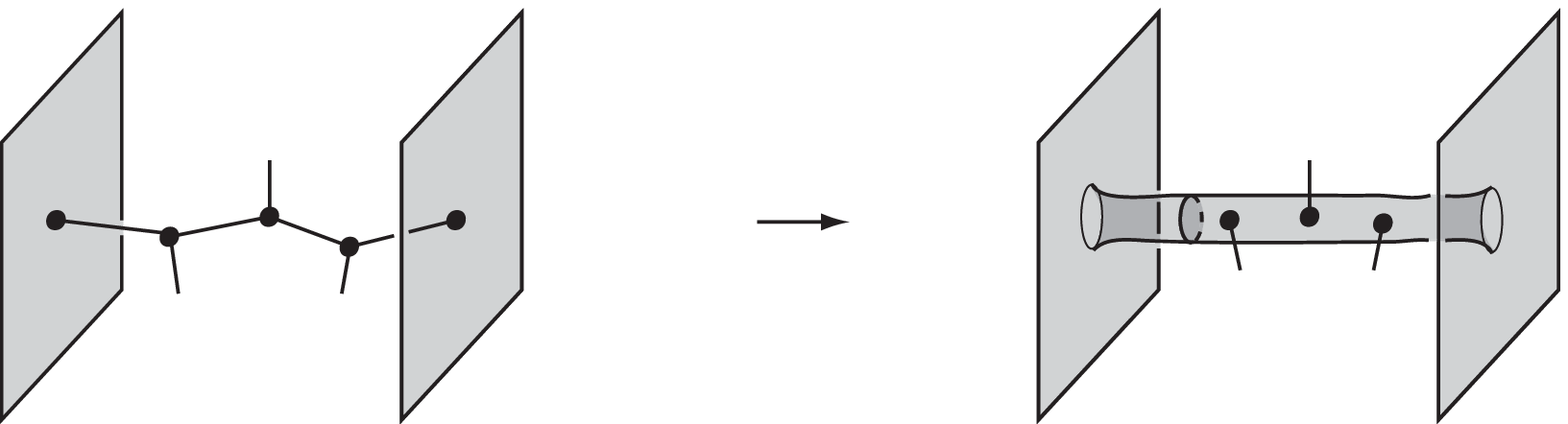}
\end{minipage}
\caption{This move reduces the number of edges of the graph part, and does not 
produce true vertices.}
\label{fig:almost-simple_to_simple1}
\end{center}
\end{figure}
Otherwise, there exists a {\it racket}, i.e. a subgraph homeomorphic to 
\[ \{ (x,y) \in \Real^2 \mid x^2 + y^2 = 1 \} \cup \{ (x, 0) \in \Real^2 \mid 1 \leq x \leq 2 \}, \]
in $Y'$ with the (unique) univalent vertex on $Y'$. 
In this case, we apply the move shown in Figure \ref{fig:almost-simple_to_simple2}. 
\begin{figure}[htbp]
\begin{center}
\begin{minipage}{12cm}
\includegraphics[width=12cm,clip]{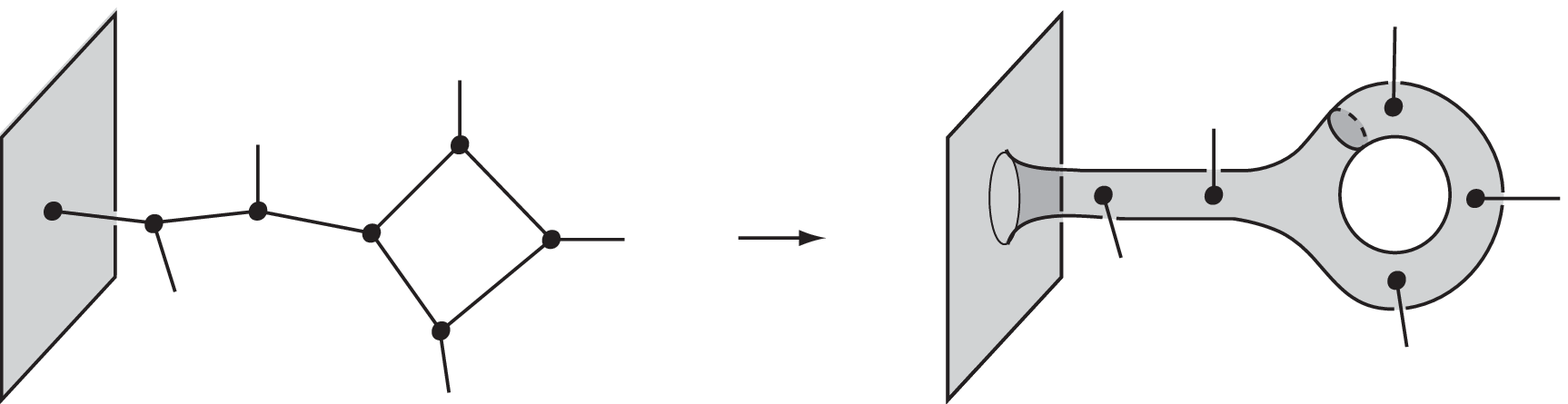}
\end{minipage}
\caption{This move reduces the number of edges of the graph part, and does not 
produce true vertices.}
\label{fig:almost-simple_to_simple2}
\end{center}
\end{figure}
Note that the polyhedra before and after each of the above two moves have the same 
regular neighborhood in $M$.
Further, the resulting polyhedron remains to be minimal with respect to collapsing, and 
the number of true vertices does not increase by each of the moves. 
Since the number of edges of $\Gamma$ is finite, by applying these moves finitely many times, 
we finally end with a closed shadow $Z$ of $M$ with at most $n$ vertices 
(see Figure \ref{fig:almost-simple_to_simple_example}). 
\begin{figure}[htbp]
\begin{center}
\begin{minipage}{12cm}
\includegraphics[width=12cm,clip]{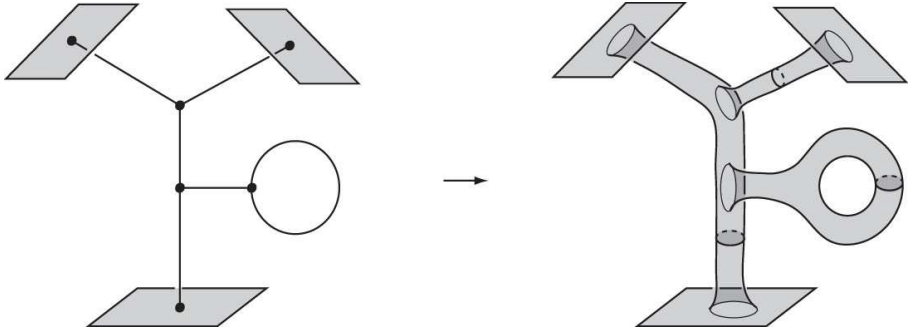}
\end{minipage}
\caption{From an almost simple polyhedron to a simple polyhedron.}
\label{fig:almost-simple_to_simple_example}
\end{center}
\end{figure}
Since the shadow-complexity of $M$ is $n$, $Z$ (and so $Y'$) has exactly 
$n$ (true) vertices.   

The argument for the connected shadow-complexity runs exactly the same way. 
\end{proof}

\begin{remark}
It is easy to see that the gleams of the small disks region produced by the moves 
in Figures \ref{fig:almost-simple_to_simple1} and \ref{fig:almost-simple_to_simple2} 
are zero. (We do not use this fact in this paper.)
\end{remark}

\subsection{Diagrams of special polyhedra}
\label{subsec:Diagrams of special polyhedra}

Let $X$ be a special polyhedron having at least one vertex. 
Recall that the singular part $S(X)$ of $X$ is a connected (possibly non-simple) 4-regular graph. 
Let $\{v_1, \ldots, v_{k}\}$ be the set of vertices and $\{e_1, \ldots, e_{2k}\}$ the set of edges of $S(X)$. 
Set $X' := \Nbd (S(X); X)$. 
Note that the closure of $X \setminus X'$ consists of disks, hence 
the topological type of $X$ is uniquely recovered from that of $X'$. 
At each vertex $v_i$, choose a neighborhood $V_i := \Nbd (v_i ; X')$ homeomorphic to 
Figure \ref{fig:simple_polyhedron} (iii) so that 
each component of the closure of $X \setminus \bigcup_{i=1}^k V_i$ is homeomorphic to 
$Y \times [0,1]$, where $Y$ is the cone over 3 points. 
Let $E_j$ be the component of $X \setminus \bigcup_{i=1}^k V_i$ corresponding to the edge $e_j$. 
We call each of $V_i$ and $E_j$ a {\it block} of $X'$. 

The diagram of $X$ is obtained as follows. 
Draw a diagram (with only normal crossings) of the graph $S(X)$ on $\Real^2$. 
In the diagram, replace  each vertex $v_i$ of $S(X)$ with 
the local diagram (which describes 
the block $V_i$) shown in Figure \ref{fig:blocks} (i). 
Replace  each edge $e_j$ of $S(X)$ in the diagram with 
one of the four local diagrams (which describes 
the block $E_j$) shown in Figure \ref{fig:blocks} (ii) so that 
the gluing of the end of the strands matches the combinatorial 
structure of $X'$. 
\begin{figure}[htbp]
\begin{center}
\begin{minipage}{12cm}
\includegraphics[width=12cm,clip]{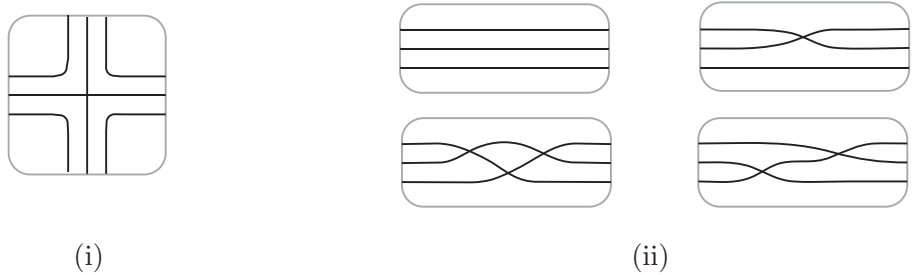}
\begin{picture}(400,0)(0,0)
\put(25,0){(i)}
\put(236,0){(ii)}
\end{picture}
\end{minipage}
\caption{Diagrams corresponding to blocks.}
\label{fig:blocks}
\end{center}
\end{figure}
 Apparently, one simple polyhedron admits many diagrams, 
but each diagram defines a unique special polyhedron. 
\begin{figure}[htbp]
\begin{center}
\begin{minipage}{8cm}
\includegraphics[width=8cm,clip]{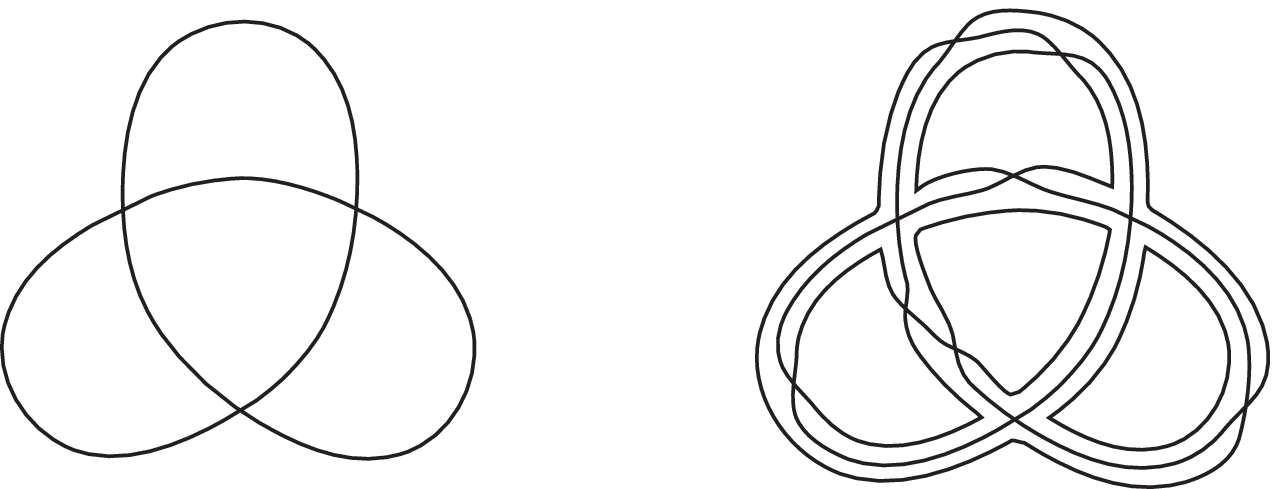}
\begin{picture}(400,0)(0,0)
\put(38,0){(i)}
\put(175,0){(ii)}
\end{picture}
\end{minipage}
\caption{An example of a diagram of a special shadow.}
\label{fig:example_shadow_diagram}
\end{center}
\end{figure}

The above decomposition of a neighborhood of the singular part of 
a special polyhedron into blocks 
allows us to enumerate all special polyhedron with a given number of vertices systematically. 
The table in Appendix \ref{sec:Table of special polyhedron of complexity up to 2} lists 
all the special polyhedra with 1 or 2 vertices. 
Note that the special polyhedron $n^{i}_j$ 
in the Appendix \ref{sec:Table of special polyhedron of complexity up to 2}  
has $n$ vertices and $i$ regions. 
The table in Appendix \ref{subsec:Special polyhedra with 1 vertices} was already given in 
\cite{KMN18} with different names of polyhedra. 
The special polyhedra $1^1_1$, $1^1_2$, $1^2_1$, $1^2_2$, $1^2_3$, $1^2_4$, $1^2_5$, $1^3_1$, $1^3_2$, $1^3_3$, $1^4_1$ 
in this paper are $X_1, X_2 , \ldots, X_{11}$ in \cite{KMN18}, respectively.


\subsection{From special shadows to Kirby diagrams}
\label{subsec:From special shadows to Kirby diagrams}

Let $X$ be a special shadow of a 4-manifold $M$. 
We can construct a Kirby diagram of $M$ as follows. 
Draw a diagram of $X$ as explained in 
Subsection \ref{subsec:Diagrams of special polyhedra}. 
The diagram can be regarded as immersed circles on $\Real^2$ 
with normal crossings.
At each crossing, choose over/under crossings in an arbitrary way. 
Choose a maximal tree $T$ of the graph $S(X)$. 
Encircle with a dotted circle the 3 strands of the diagram corresponding to 
each edge of $S(X)$ not contained in $T$.

\begin{figure}[htbp]
\begin{center}
\begin{minipage}{9cm}
\includegraphics[width=9cm,clip]{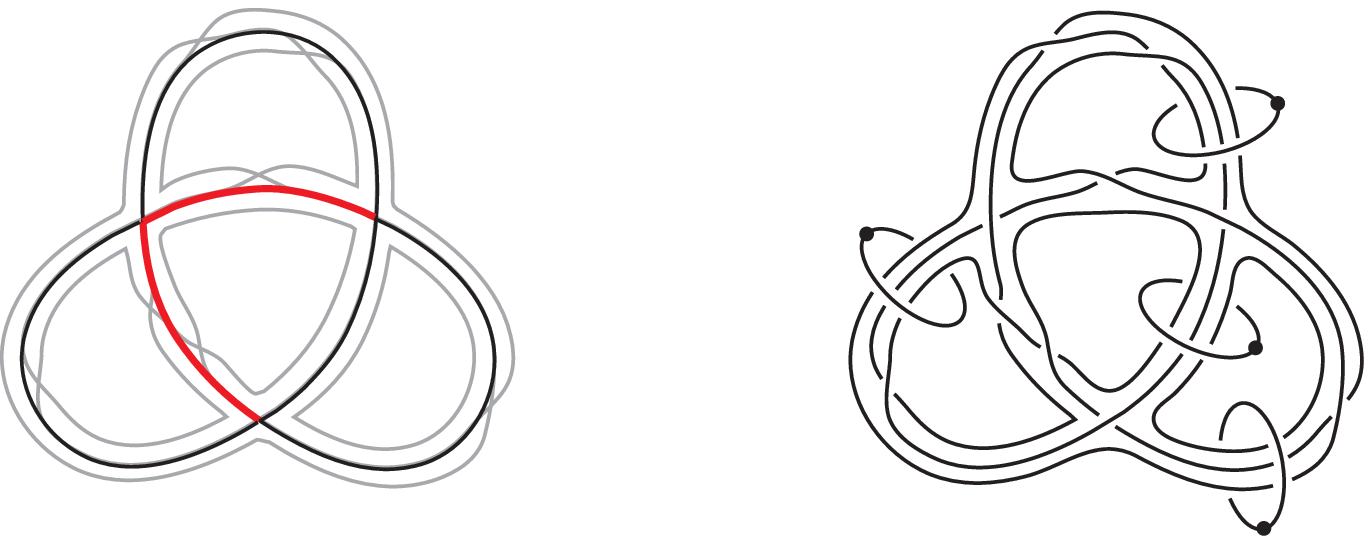}
\begin{picture}(400,0)(0,0)
\put(42,0){(i)}
\put(201,0){(ii)}
\put(36,64){\color{red} $T$}
\end{picture}
\end{minipage}
\caption{From a diagram of a shadow to a Kirby diagram.}
\label{fig:shadow_to_kirby}
\end{center}
\end{figure}

Let $X$ be a special polyhedron with $m$ regions. 
For each region $R$ of $X$, let $\iota_R : R \hookrightarrow M$ be the inclusion.  
Let $\bar{R}$ be the metric completion of $R$ with the path metric inherited from a Riemanian metric on $R$. 
Let $\bar{\iota}_R : \bar{R} \to M$ be the natural extension of $\iota$. 
Let $T$ be a maximal tree in $S(X)$.  

\begin{definition}
We say that $X$ admits $k$ $(\leq m)$ {\it canceling pairs} (with respect to $T$) 
if there exist
\begin{itemize} 
\item
an ordered subset $\{ e_1, \ldots, e_k \}$ of the edges of $S(X)$ not contained in $T$; and 
\item
an ordered subset $\{ R_1, \ldots, R_k \}$ of the regions of $X$  
\end{itemize} 
such that for each $i \in \{ 1, \ldots, k \}$, 
$\partial \bar{R}_i$ passes through $e_i$ exactly once, and 
$\partial \bar{R}_i$ does not pass through $e_j$ (for $i < j \leq k$). 
\end{definition}
Remark that if a special polyhedron $X$ admits $k$ canceling pairs with respect to $T$ then 
a Kirby diagram obtained from $X$ by using $T$ as above 
admits $k$ canceling pairs of $1$- and $2$-handles (a dotted circle and a component of the framed link). 

\subsection{Graphs encoding simple polyhedra without vertices}
\label{subsec:Graphs encoding simple polyhedra without vertices}

Let $Y$ be the cone over 3 points. 
We denote by $Y_{111}$, $Y_{12}$, $Y_3$ the three 
$Y$-bundles over $S^1$ such that 
$\# \partial Y_{111} = 3$, 
$\# \partial Y_{12} = 2$, 
$\# \partial Y_{3} = 1$.

Every simple polyhedron $X$ whose singular set is a disjoint union of circles is decomposed into pieces 
each homeomorphic to 
a disk $D^2$, a pair of pants $P$, a M\"obius band $Y_2$, $Y_{111}$, $Y_{12}$ or $Y_3$. 
Such a decomposition induces a graph 
having one vertex for each piece or a component of $\partial X$ as shown in 
Figure \ref{fig:graphic_encoding}, 
and one edge for each adjacent pieces. 
\begin{figure}[htbp]
\begin{center}
\begin{minipage}{12cm}
\includegraphics[width=12cm,clip]{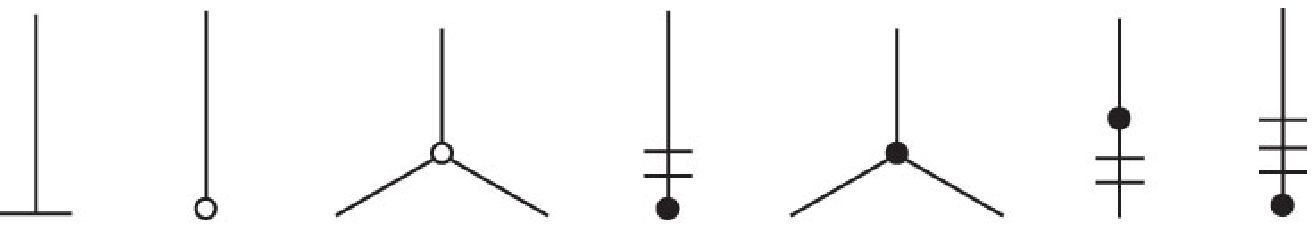}
\begin{picture}(400,0)(0,0)
\put(2,0){(B)}
\put(45,0){(D)}
\put(108,0){(P)}
\put(167,0){(2)}
\put(221,0){(111)}
\put(283,0){(12)}
\put(327,0){(3)}
\end{picture}
\end{minipage}
\caption{Vertices of a graph encoding the pieces 
$B$, $D^2$, $P$, $Y_2$, $Y_{111}$, $Y_{12}$ and $Y_3$
of a simple polyhedron, where $B$ is a component of the boundary.}
\label{fig:graphic_encoding}
\end{center}
\end{figure}
This graph is introduced by Martelli in \cite{Mar11} for the classification of closed 4-manifold 
with shadow-complexity 0. 
Let $G$ be a graph encoding $X$. 
We note that there is a natural (but not unique) embedding $G \hookrightarrow X$ such that 
$G$ is a retract of $X$. 
In particular, we have the following: 
\begin{lemma}
\label{lem:retraction}
Let $G$ be a graph encoding a simple polyhedron $X$ whose singular set is a disjoint union of circles. 
Then a natural embedding $G \hookrightarrow X$ induces an injection of 
the fundamental groups. 
\end{lemma}

As was mentioned in \cite{Mar11}, the simple polyhedron $X$ is recovered from a pair of 
a graph $G$ and a map $\beta : H_1(G; \Integer / 2 \Integer) \to \Integer / 2 \Integer$, 
which describes each cycle in $G$ is orientation-preserving or not after naturally embedding 
$G$ into $X$ (this property does not depend of the choice of embedding).


\section{Main Theorem}
\label{sec:Main Theorem}

We prove the following. 

\begin{theorem}
\label{thm:acyclic 4-manifold of connected shadow-complexity at most 2}
Every acyclic $4$-manifold of connected shadow-complexity at most $2$  
with boundary the $3$-sphere is diffeomorphic to the standard $4$-ball. 
In other words, there exists no acyclic $4$-manifold of connected shadow-complexity $1$ or $2$  
with boundary the $3$-sphere. 
\end{theorem}

The key ingredients in the proof are Property R theorem by Gabai \cite{Gab87} and 
detailed analyses of acyclic polyhedra.  
In fact we prove the same thing as above in a more general setting in Theorem \ref{thm:general setting}. 
We note that the nature of acyclic polyhedra with vertices are completely different from those 
without vertices. 
In \cite{Nao17} it was shown that every acyclic simple polyhedron without vertices collapses onto $D^2$. 
In contrast, as we will see in the 
following arguments, there exists infinitely many closed acyclic simple polyhedra for a given number of vertices. 

\subsection{Acyclic special polyhedra}
\label{subsec:Acyclic special polyhedra}

In this subsection, we focus on special polyhedra, and 
prove a special case 
(Lemma \ref{lem:acyclic 4-manifold of special shadow-complexity at most 2}) 
of Theorem \ref{thm:acyclic 4-manifold of connected shadow-complexity at most 2}.

\begin{lemma}
\label{lem:number of regions of acyclic special polyhedra}
Let $X$ be an acyclic special polyhedron with $n$ vertices. 
Then the number of regions of $X$ is exactly $n+1$. 
\end{lemma}
\begin{proof}
Since $X$ is acyclic, the Euler characteristic $\chi (X)$ is one. 
The number of edges of $X$ is $2n$ as $S(X)$ is a $4$-regular graph. 
Thus, by the assumption that $X$ is special, we have 
\[ 1 = \chi (X) = \# V(X)  - \# E(X) + \# R(X) = n - 2n + \# R(X), \]
which implies $\# R(X) = n + 1$. 
\end{proof}

\begin{lemma}
\label{lem:acyclic simply-connected special polyhedron}
An acyclic 
special polyhedron with 
 at most $2$ vertices is one of 
$1^2_i$ $(i=1,2)$ and 
$2^3_i$ $(i=1,2, \ldots, 16)$. 
\end{lemma}
\begin{proof}
This is easily checked by Lemma 
\ref{lem:number of regions of acyclic special polyhedra} and 
the table in Appendixes 
\ref{subsec:Special polyhedra with 1 vertices} and  
\ref{subsec:Special polyhedra with 2 vertices and 3 regions}. 
\end{proof}

Recall the following famous Property R theorem by Gabai \cite{Gab87}. 

\begin{theorem}[Gabai \cite{Gab87}]
\label{thm:Property R}
Any $3$-manifold obtained by $0$-surgery on a nontrivial knot in $S^3$ is irreducible.
In particular, non-trivial surgery on a non-trivial knot in $S^3$ does knot yield $S^2 \times S^1$. 
\end{theorem}

Theorem \ref{thm:Property R} implies that 
a Kirby diagram with a (framed) knot and a dotted circle 
of a 4-manifold $M$ with $\partial M \cong S^3$ is 
nothing but the one shown in Figure \ref{fig:property_R}. 
In particular, $M$ is diffeomorphic to $D^4$. 
\begin{figure}[htbp]
\begin{center}
\begin{minipage}{3cm}
\includegraphics[width=3cm,clip]{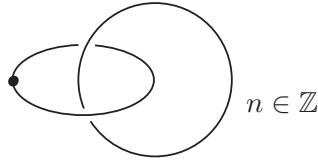}
\begin{picture}(400,0)(0,0)
\put(90,30){$n\in\Integer$}
\end{picture}
\end{minipage}
\caption{A Kirby diagram of $D^4$.}
\label{fig:property_R}
\end{center}
\end{figure}
The following is the direct consequence of Theorem \ref{thm:Property R}. 
\begin{lemma}
\label{lem:a sufficient condition to be the disk in terms of special shadow}
Let $M$ be an acyclic $4$-manifold with $\partial M \cong S^3$. 
Let $X$ be a special shadow of $M$ with $n$ vertices. 
If $X$ admits $n$ canceling pairs, $M$ is diffeomorphic to $D^4$. 
\end{lemma}
\begin{proof}
By Lemma \ref{lem:number of regions of acyclic special polyhedra}
a Kirby diagram of $M$ corresponding to $X$ is 
an $(n+1)$-component (framed) link $L$ with $n+1$ dotted circles. 
If that diagram admits $n$ cancelling pairs of $1$- and $2$-handles, 
after canceling them, we get a (framed) knot $K$ with a dotted circle. 
Hence, the assertion follows from Theoerm \ref{thm:Property R}. 
\end{proof}

\begin{lemma}
\label{lem:acyclic 4-manifold of special shadow-complexity at most 2}
Every acyclic $4$-manifold $M$ of special shadow-complexity at most $2$ 
with $\partial M \cong S^3$ is diffeomorphic to $D^4$. 
\end{lemma}
\begin{proof}
The case of special shadow-complexiy 0 is due to 
Theorem \ref{thm:Costantino and Naoe}
(\ref{thm:acyclic special polyhedron of complexity 0}). 
Let $M$ be an acyclic $4$-manifold with special shadow-complexity $n$, where $n$ is $1$ or $2$. 
Let $X$ be a special shadow of $M$ with $n$ vertices. 
By Lemma \ref{lem:acyclic simply-connected special polyhedron}, 
$X$ is one of 
$1^2_i$ ($i=1,2$) and 
$2^3_i$ ($i=1, 2, \ldots, 16$). 
For each case, we can easily check that all of them admit 
$n$ canceling pairs. 
Thus, by Lemma \ref{lem:a sufficient condition to be the disk in terms of special shadow},  
$M$ is diffeomorphic to $D^4$. 
\end{proof}

\begin{remark}
It is worth noting that every shadow $X$ of a $4$-manifold $M$ with $\partial M = S^3$ 
is simply-connected. 
In fact, the restriction of a projection $M \searrow X$ to $\partial M$ 
induces a surjective homomorphism $\pi_1(\partial M) \to \pi_1(X)$. 
This fact does not give any restriction for the shadows in the above argument. 
That is, every acyclic special polyhedron with vertices up to $2$ is simply-connected. 
It is an interesting problem to find an acyclic special shadow that is not 
simply-connected. 
\end{remark}

\subsection{Acyclic simple polyhedra}
\label{subsec:Acyclic simple polyhedra}

We are going to extend Lemma \ref{lem:acyclic 4-manifold of special shadow-complexity at most 2} 
to the 4-manifolds of connected shadow-complexity at most $2$. 
We begin with two lemmas that will be used repeatedly in the remaining part of the paper.  

\begin{lemma}[Ikeda \cite{Ike71}]
\label{lem:Regions are planar}
Let $X$ be an acyclic simple polyhedron. 
Then the following holds.  
\begin{enumerate}
\item
Every region of $X$ is a $2$-sphere with holes. 
\item
If $X$ has no vertices, then $\partial X \neq \emptyset$. 
\end{enumerate}
\end{lemma}

\begin{lemma}[Naoe \cite{Nao17}]
\label{lem:Naoe Lem 3.2}
Let $X$ be an acyclic simple polyhedron. 
Let $\gamma \subset X \setminus S(X)$ be a simple closed curve. 
Then $\gamma$ splits $X$ into $A$ and $B$ such that 
\begin{enumerate}
\item
$A$ is acyclic; and 
\item
$B$ is a homology-$S^1$ and $H_1(B; \Integer)$ is generated by the cycle represented by $\gamma$.  
\end{enumerate}
\end{lemma}

We note that if the piece $B$ in Lemma \ref{lem:Naoe Lem 3.2} has no vertices, 
then it has at least one boundary component other than $\gamma$. 
Otherwise, the polyhedron $\hat{B}$ obtained from $B$ by capping off $\gamma$ 
by a disk is an acyclic polyhedron with $\# V(\hat{B}) = 0$ and $\partial \hat{B} = \emptyset$, 
which contradicts Lemma \ref{lem:Regions are planar}. 
In particular, an acyclic simple polyhedron does not contain a piece homeomorphic to 
$Y_2$ nor $Y_3$ (recall Subsection \ref{subsec:Graphs encoding simple polyhedra without vertices}). 

The following lemma is a generalization of 
Lemma \ref{lem:number of regions of acyclic special polyhedra}. 

\begin{lemma}
\label{lem: decomposition of X in the case where a component of S(X) contains all vertices}
Let $X$ be an acyclic simple polyhedron with at least one vertex. 
Suppose that there exists a component $S'$ of $S(X)$ containing $n$ $(\geq 1)$ vertices. 
Set $X' := \Nbd (S'; X)$ and $m := \# \partial X'$. 
Then $m \geq n+1$ and the simple closed curves $\partial X'$ splits $X$ into 
$X'$, $n+1$ acyclic pieces $A_1, \ldots, A_{n+1}$, and $m-n-1$ homology-$S^1$ pieces $B_{1} , \ldots , B_{m-n-1}$. 
Further, if $X$ is closed, each $B_i$ contains at least one vertex. 
\end{lemma}
\begin{proof}
By Lemma \ref{lem:Naoe Lem 3.2} 
and the Euler characteristic computation, 
it is straightforward to see that  
$\partial X'$ splits $X$ into 
$X'$, $n+1$ acyclic pieces $A_1, \ldots, A_{n+1}$, and $m-n-1$ homology-$S^1$ pieces $B_{1} , \ldots , B_{m-n-1}$. 
Suppose that some $B_i$ does not contain vertices. 
Then as we noted right after Lemma \ref{lem:Naoe Lem 3.2}, 
$\partial B_i$ consists of at least two components. 
This implies that $\partial X$ is not empty. 
\end{proof}

\begin{remark}
If an acyclic piece $A_i$ in 
Lemma \ref{lem: decomposition of X in the case where a component of S(X) contains all vertices} 
contains no vertices, 
we can describe its explicit shape as we discuss in  
Appendix \ref{sec:Acyclic simple polyhedron without vertices and with a single boundary circle}. 
\end{remark}

\begin{lemma}
\label{lem: the case where a component of S(X) contains all vertices}
Let $M$ be an acyclic $4$-manifold with $\partial M \cong S^3$. 
Let $X \subset M$ be a shadow with $n$ vertices. 
Suppose that there exists a component $S'$ of $S(X)$ containing all vertices of $X$. 
Set $X' := \Nbd (S'; X)$. 
Suppose that $\# \partial X' = n+1$. 
Let $\hat{X}'$ be the special polyhedron obtained from $X'$ 
by capping off the boundary components by disks. 
If $\hat{X}'$ admits $n$ canceling pairs, then $M $ is diffeomorphic to $D^4$.
\end{lemma}
\begin{proof}
By Lemma \ref{lem: decomposition of X in the case where a component of S(X) contains all vertices}, 
the simple closed curves $\partial X'$ splits $X$ into 
$X'$ and $n+1$ acyclic pieces $A_1, \ldots, A_{n+1}$. 
Set $\gamma_i = X' \cap A_i$ for $i \in \{ 1 , \ldots , n+1 \}$. 
The decomposition of $X$ into $X'$ and $A_1, \ldots, A_{n+1}$ 
naturally induces a decomposition of $M$. 
Let $M_{X'}$ and $M_{A_i}$ be the pieces of the decomposition of $M$ corresponding to $X'$ and $A_i$. 
Note that $M_{X'}$ is diffeomorphic to $\natural_{n+1} (S^1 \times D^3)$. 
Since $A_i$ does not contain vertices,  
$M_{A_i}$ is diffeomorphic to $D^4$ by Theorem \ref{thm:Costantino and Naoe} 
(\ref{thm:acyclic simple polyhedron of complexity 0}). 
Let $K_i$ be the framed knot in $\partial M_{A_i}$ $(\cong S^3)$ corresponding to 
$\gamma_i$. 

Since $\hat{X}'$ admits $n$ canceling pairs, 
some Kirby diagram obtained from a diagram of $\hat{X}'$ as explained in Section 
\ref{subsec:From special shadows to Kirby diagrams} admits $n$ canceling pairs. 
That Kirby diagram consists of 
an $n+1$-component framed link $L = L_1 \sqcup \cdots \sqcup L_{n+1}$, 
where $L_i$ corresponds to $\gamma_i$, 
together with $n+1$ dotted
circles $U = U_1 \sqcup \cdots \sqcup U_{n+1}$. 
By Lemmas \ref{lem:connected sum of framed knots} and 
\ref{lem:gluing formula for shadows}, 
the Kirby diagram obtained from $L \cup U$ by replacing 
$L_i$ with $L_i \# K_i$ represents the 4-manifold $M$. 
Hence, that Kirby diagram of $M$ can be 
simplified to a (framed) knot $K$ with a single dotted circle by handle-canceling.
Since $\partial M \cong S^3$, the Kirby diagram thus obtained is the one shown in Figure \ref{fig:property_R} 
by Theorem \ref{thm:Property R}. 
This implies that $M$ is diffeomorphic to $D^4$.  
\end{proof}

Let $X$ be a closed acyclic simple polyhedron. 
We say that $X$ satisfies the \textit{cancellation condition} if 
the following holds: For each component $S'$ of $S(X)$ such that 
$\# V (X') \geq 1$ and $\# \partial X' = \# V (X') + 1$, 
where $X' := \Nbd (S'; X)$, 
the special polyhedron $\hat{X}'$ obtained from $X'$ 
by capping off the boundary components by disks admits 
$\# V (X')$ canceling pairs.

Theorem \ref{thm:acyclic 4-manifold of connected shadow-complexity at most 2} is a direct consequence 
of the following theorem.

\begin{theorem}
\label{thm:general setting}
Let $M$ be an acyclic $4$-manifold with $\partial M \cong S^3$. 
If $M$ admits a closed shadow satisfying the cancellation condition, 
then $M$ is diffeomorphic to $D^4$. 
\end{theorem}

\begin{proof}[Proof of Theorem $\ref{thm:acyclic 4-manifold of connected shadow-complexity at most 2}$ from Theorem 
$\ref{thm:general setting}$]
Let $M$ be an acyclic $4$-manifold of connected shadow-complexity $n$   
with $\partial M \cong S^3$, where $n = 0, 1$ or $2$. 
The case where $n=0$ is due to 
Theorem \ref{thm:Costantino and Naoe} 
(\ref{thm:acyclic simple polyhedron of complexity 0}). 
In the following we assume that $n = 1$ or $2$. 
By Lemma \ref{lem:non-collapsible polyhedra}, 
$M$ admits a closed shadow $X$ of connected complexity $n$. 
Let $S'$ be a component of $S(X)$ such that 
$\# V (X') \geq 1$ and $\# \partial X_i = \# V (X') + 1$, 
where $X' := \Nbd (S'; X)$. 
Then the special polyhedron $\hat{X}'$ obtained from $X'$ 
by capping off the boundary components by disks remains to be acyclic. 
As we have seen in the proof of 
Lemma \ref{lem:acyclic 4-manifold of special shadow-complexity at most 2}, 
$\hat{X}'$ admits $\# V (X') + 1$ canceling pairs. 
Therefore, $X$ satisfies the cancellation condition. 
Consequently, $M$ is diffeomorphic to $D^4$ by Theorem \ref{thm:general setting}. 
\end{proof}

\begin{proof}[Proof of Theorem $\ref{thm:general setting}$]
Let $M$ be an acyclic $4$-manifold with $\partial M \cong S^3$.  
Let $X$ be a closed shadow of $X$ satisfying the cancellation condition. 
If $X$ contains no vertices, 
then the assertion follows from  
Theorem \ref{thm:Costantino and Naoe} 
(\ref{thm:acyclic simple polyhedron of complexity 0}). 
In the following we assume that $X$ contains vertices. 
Let $S_1, S_2, \ldots, S_k$ be the connected components of 
$S(X)$ having at least one vertex. 
We use the induction on $k$. 

Let $k = 1$. 
Set $X' := \Nbd (S_1; X)$. 
By Lemma \ref{lem: decomposition of X in the case where a component of S(X) contains all vertices}, 
we have $\# \partial X' = \# V (X') + 1$. 
Since $X$ satisfies the cancellation condition, $M$ is diffeomorphic to $D^4$ 
by Lemma \ref{lem: the case where a component of S(X) contains all vertices}.

Let $k \geq 2$ and assume that the conclusion holds for all $k' < k$. 
Set $X_i := \Nbd (S_i; X)$ and $n_i := \# V(X_i)$ for $i=1,2, \ldots, k$. 
By Lemma \ref{lem: decomposition of X in the case where a component of S(X) contains all vertices}, 
for each $i$, the simple closed curves $\partial X_i$ splits $X$ into 
$X'$, $n_i + 1$ acyclic pieces, and several (possibly no) homology-$S^1$ pieces. 
Further, here if there exists a homology-$S^1$ piece, 
then it contains a vertex. 
The argument is divided into two cases. 

\noindent 

\noindent {\it Case} 1: 
Suppose first that there exists $i \in \{ 1, 2, \ldots, k\}$ such that $\# \partial X_i = n_i + 1$. 
Without loss of generality, we can assume that $\# \partial X_1 = n_1 + 1$. 
Note that $n_1 = 1$ or $2$. 
By Lemma \ref{lem: decomposition of X in the case where a component of S(X) contains all vertices}, 
$\partial X_1$ splits $X$ into 
$X_1$ and $n_1 + 1$ acyclic pieces $A_1, \ldots, A_{n_1+1}$. 
The decomposition of $X$ into $X_1$, $A_1, \ldots, A_{n_1+1}$  
naturally induces a decomposition of $M$ into $M_{X_1}$, $M_{A_1}, \ldots, M_{A_{n_1+1}}$. 
Since the number of connected components of $A_j$ ($j \in \{ 1, \ldots, n_1 + 1 \}$) 
having at least 1 vertex is fewer than $k$, we have $M_{A_j} \cong D^4$ for each 
$j \in \{ 1, \ldots, n_1 + 1 \}$ 
by the assumption of the induction. 
Now the rest of the proof for this case runs as in 
Lemma \ref{lem: the case where a component of S(X) contains all vertices}. 

\noindent {\it Case} 2: 
Suppose that $\# \partial X_i > n_i + 1$ for all $i \in \{ 1, 2, \ldots, k\}$. 
We are going to show that in this case $X$ is not closed, which is a contradiction. 
Let $Y_1, Y_2, \ldots, Y_l$ be the connected components of 
$\overline{X - \bigcup_{i=1}^k X_i}$. 
Let $\hat{G}$ be the bipartite graph whose vertices are 
$\{ X_1, X_2, \ldots, X_k \} \sqcup \{ Y_1, Y_2, \ldots, Y_l \} $ such that two vertices 
$X_i$ and $Y_j$ span an edge 
if and only if $X_i \cap Y_j \neq \emptyset$. 
Note that since $X$ is acyclic, $\hat{G}$ is a tree. 
Note also that the set of edges of $\hat{G}$ one to one corresponds to 
the set of simple closed curves of $\bigcup_{i=1}^k \partial X_i$. 
Each edge of $\hat{G}$ is labeled by $0$ or $1$ as follows.  
Let $e$ be the edge of $\hat{G}$ corresponding to 
a simple closed curve $\gamma$ of $\partial X_i$. 
The curve $\gamma$ separates $X$ into two components 
$X_\gamma$ and $Y_\gamma$, where $X_i \subset X_\gamma$. 
See Figure \ref{fig:case2}. 
\begin{figure}[htbp]
\begin{center}
\begin{minipage}{6cm}
\includegraphics[width=6cm,clip]{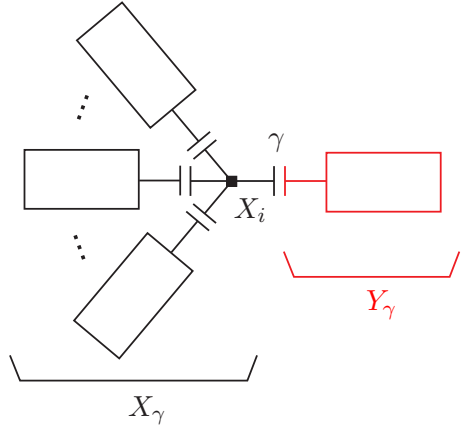}
\begin{picture}(400,0)(0,0)
\put(45,0){$X_\gamma$}
\put(135,40){\color{red} $Y_\gamma$}
\put(85,75){$X_i$}
\put(98,102){$\gamma$}
\end{picture}
\end{minipage}
\caption{The decomposition of $X$ into $X_\gamma$ and $Y_\gamma$.}
\label{fig:case2}
\end{center}
\end{figure}
By Lemma \ref{lem:Naoe Lem 3.2}, $Y_\gamma$ is  
acyclic or a homology-$S^1$. 
We assign $0$ to $e$ if $Y_\gamma$ is acyclic, and $1$ to $e$ if $Y_\gamma$ is a homology-$S^1$.  
Note that for each vertex $X_i$, there exists 
at least one $1$-labeled edge connected to $X_i$ by the assumption of Case 2. 
\begin{claim*}
\label{claim}
The graph $\hat{G}$ contains a vertex $Y_i$ such that every edge connected to $Y_i$ is 
labeled by $1$. 
\end{claim*}
\begin{proof}[Proof of Claim]
Suppose for a contradiction that for each vertex $Y_i$, 
there exists a $0$-labeled edge connected to $Y_i$. 
Choose an arbitrary vertex $X_{k_1}$ of $\hat{G}$. 
Then $X_{k_1}$ is connected to a vertex $Y_{l_1}$ of $\hat{G}$ by a $1$-labeled edge. 
By the assumption, $Y_{l_1}$ is connected to a vertex $X_{k_2}$ of $\hat{G}$ by a $0$-labeled edge. 
Then $X_{k_2}$ is connected to a vertex $Y_{l_2}$ by a $1$-labeled edge. 
In this way, we can make a path of arbitrary large length. 
During the process, we do not pass the same vertex more than once because $\hat{G}$ is a tree. 
This contradicts the finiteness of $\hat{G}$. 
\end{proof}
By the above claim, without loss of generality, we can assume that 
the all edges, say $\gamma_1 , \ldots, \gamma_m$, connected to $Y_1$ is labeled by $1$. 
This implies that for each $i$, 
\begin{itemize} 
\item
$X_{\gamma_i}$ is acyclic; and  
\item
$Y_{\gamma_i}$ is a homology-$S^1$ and 
$H_1(Y_{\gamma_i}; \Integer)$ is generated by the cycle represented by $\gamma_i$. 
\end{itemize} 
Let $Z_1$ be the polyhedron obtained from $Y_{\gamma_1}$ by capping off $\gamma_1$ 
by a disk. 
Then $Z_1$ is acyclic. 
Construct inductively a sequence of polyhedrons $Z_1, Z_2, \ldots, Z_m$, where 
$Z_{i+1}$ is the polyhedron obtained from $\overline{Z_i \setminus X_{\gamma_{i+1}}}$ 
by capping off $\gamma_{i+1}$ by a disk. 
Since each $X_{\gamma_i}$ is acyclic, $Z_i$ is again acyclic. 
Since $Z_m$ contains no vertex, by Lemma \ref{lem:Regions are planar} (2), 
$Z_m$ has at least one boundary component. 
This implies that $X$ is not closed, which is a contradiction. 
\end{proof}

\section*{Acknowledgments} 
The authors wish to express their gratitude to 
Kouichi Yasui for many helpful comments.  

\appendix 

\section{Acyclic simple polyhedron without vertices and with a single boundary circle.}
\label{sec:Acyclic simple polyhedron without vertices and with a single boundary circle}

Let $X$ be an acyclic simple polyhedron without vertices and with a single boundary circle $\gamma$. 
We can describe a specific shape of $X$ as follows. 
Since $X$ is acyclic, $X$ cannot contain a piece homeomorphic to $Y_2$ nor $Y_3$.  
Further, $X$ does not contain a piece homeomorphic to $Y_{111}$ as well by 
the following Lemma. 

\begin{lemma}[Naoe \cite{Nao17}]
\label{lem:collapsing respecting a fixed boundary}
Let $X$ be an acyclic simple polyhedron without vertices. 
Fix a component $\gamma$ of $\partial X$. 
Then $X$ collapses onto a sub-polyhedron $X'$ fixing $\gamma$ such that 
$X'$ does not contain a piece homeomorphic to  $Y_{111}$ and $\partial X' = \gamma$. 
\end{lemma}

Let $G$ be a graph $G$ encoding $X$, which is a tree by Lemma \ref{lem:retraction}. 
Let $v_0$ be the unique vertex of type (B) (recall Figure \ref{fig:graphic_encoding}),  
which corresponds to the unique boundary component $\partial X = \gamma$, in $G$. 
Let $v_1$ be a vertex of type $(12)$ in $G$. 
Since $G$ is a tree, there exists a unique path from $v_0$ to $v_1$. 
Let $e$ be an edge in the path incident to $v_1$. 
\begin{lemma}
\label{lem:vertex of type (12)}
In the above setting, the edge $e$ is marked with two lines. 
\end{lemma}
\begin{proof}
Suppose not for a contradiction. 
By the simple closed curve $\alpha$ corresponding to $e$, 
the polyhedron $X$ decomposes into 2 polyhedra $X_1$ and $X_2$, where $X_1$ contains $\gamma$. 
We note that $\partial X_2 = \alpha$. 
By Lemma \ref{lem:Naoe Lem 3.2}, one of them is acyclic, 
and the other is a homology-$S^1$. 
Since $ X_2$ has no boundary component other than $\alpha$, $X_2$ cannot be a homology-$S^1$. 
Thus $X_2$ is acyclic. 
By collapsing $X_2$ from $\alpha$, we obtain an acyclic simple polyhedron containing 
a M\"obius band in a region. This contradicts Lemma \ref{lem:Regions are planar}.
\end{proof}
Now we are ready to describe the shape of $X$. 
For convenience, as a generalization of a vertex of type (P), we introduce a white vertex 
of degree $d \geq 3$ as shown in Figure \ref{fig:acyclic_piece} (i) 
to encode a piece of a simple polyhedron homeomorphic to the sphere with $d$ holes. 
By the previous observation and Lemma \ref{lem:vertex of type (12)}, 
the shape of $G$ can thus be described as in 
Figure \ref{fig:acyclic_piece} (ii).  
\begin{figure}[htbp]
\begin{center}
\begin{minipage}{14cm}
\includegraphics[width=14cm,clip]{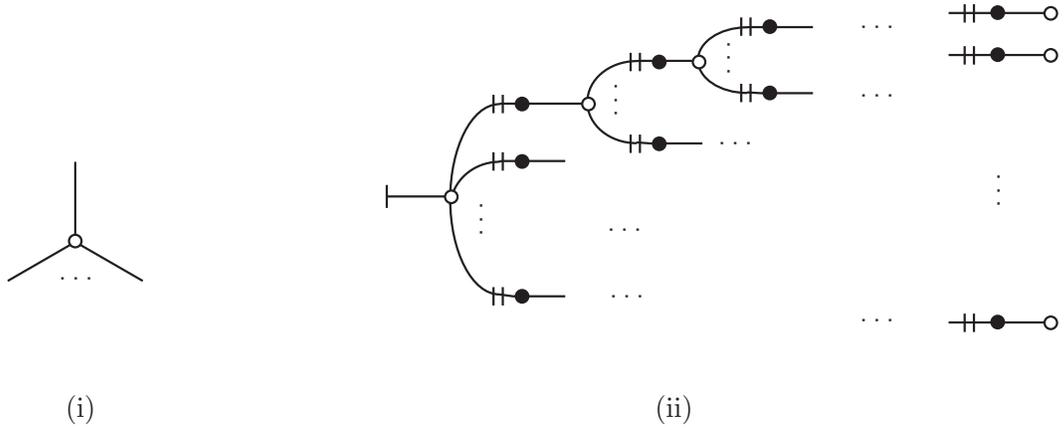}
\begin{picture}(400,0)(0,0)
\put(22,0){(i)}
\put(245,0){(ii)}
\end{picture}
\end{minipage}
\caption{(i) A vertex encoding a sphere with a finite number of holes. (ii) The graph $G$.}
\label{fig:acyclic_piece}
\end{center}
\end{figure}

\section{Table of special polyhedron of complexity up to $2$}
\label{sec:Table of special polyhedron of complexity up to 2}

\subsection{Special polyhedra with $1$ vertices}
\label{subsec:Special polyhedra with 1 vertices}

\begin{center}
$1^1_{1}$\parbox[c]{26mm}{\includegraphics[width=22mm]{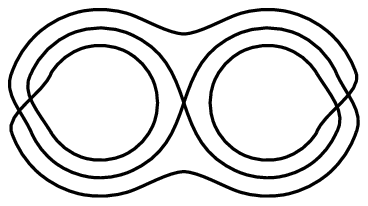}}
$1^1_{2}$\parbox[c]{26mm}{\includegraphics[width=22mm]{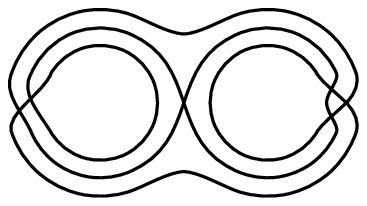}}
$1^2_{1}$\parbox[c]{26mm}{\includegraphics[width=22mm]{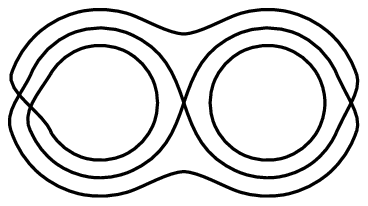}}
$1^2_{2}$\parbox[c]{26mm}{\includegraphics[width=22mm]{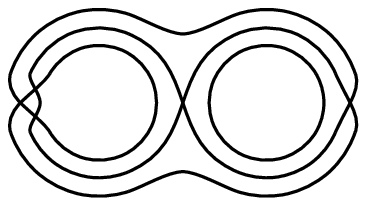}}
$1^2_{3}$\parbox[c]{26mm}{\includegraphics[width=22mm]{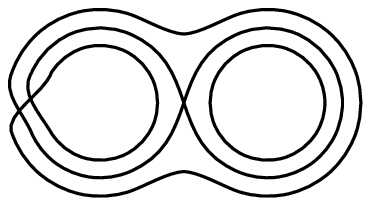}}
$1^2_{4}$\parbox[c]{26mm}{\includegraphics[width=22mm]{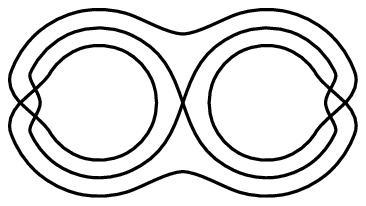}}
$1^2_{5}$\parbox[c]{26mm}{\includegraphics[width=22mm]{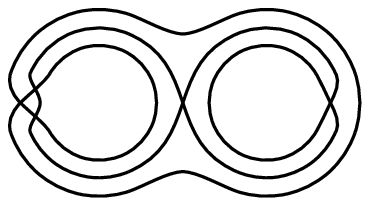}}
$1^3_{1}$\parbox[c]{26mm}{\includegraphics[width=22mm]{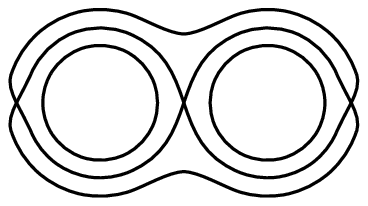}}
$1^3_{2}$\parbox[c]{26mm}{\includegraphics[width=22mm]{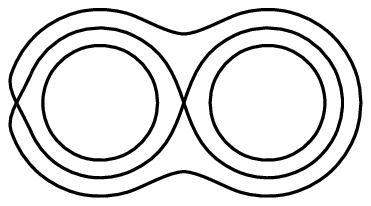}}
$1^3_{3}$\parbox[c]{26mm}{\includegraphics[width=22mm]{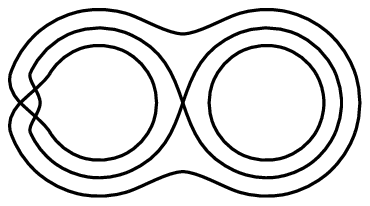}}
$1^4_{1}$\parbox[c]{26mm}{\includegraphics[width=22mm]{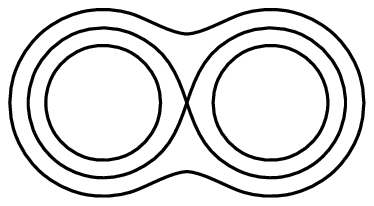}}
\end{center}

\subsection{Special polyhedra with $2$ vertices and $1$ region}

\begin{center}
$2^1_{1}$\parbox[c]{36mm}{\includegraphics[width=32mm]{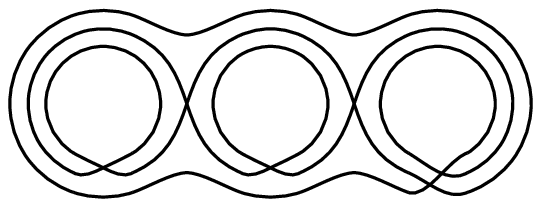}}
$2^1_{2}$\parbox[c]{36mm}{\includegraphics[width=32mm]{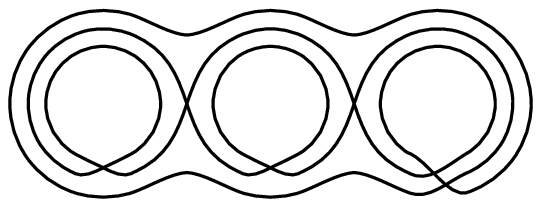}}
$2^1_{3}$\parbox[c]{36mm}{\includegraphics[width=32mm]{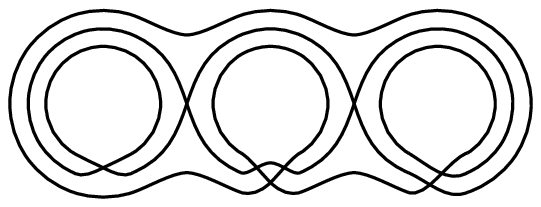}}
$2^1_{4}$\parbox[c]{36mm}{\includegraphics[width=32mm]{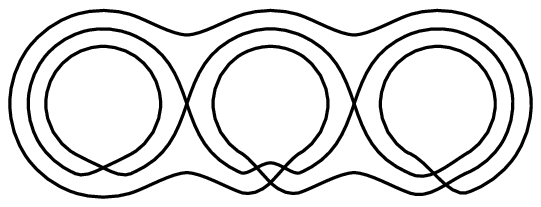}}
$2^1_{5}$\parbox[c]{36mm}{\includegraphics[width=32mm]{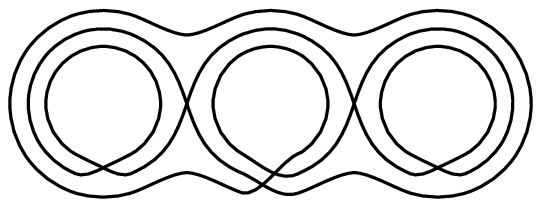}}
$2^1_{6}$\parbox[c]{36mm}{\includegraphics[width=32mm]{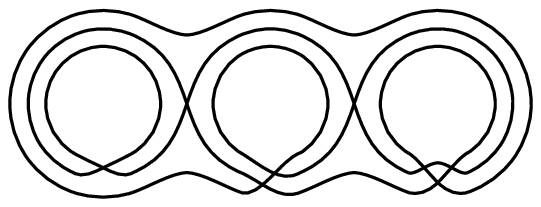}}
$2^1_{7}$\parbox[c]{36mm}{\includegraphics[width=32mm]{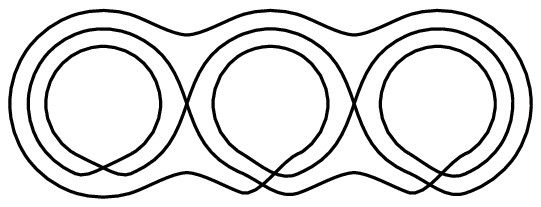}}
$2^1_{8}$\parbox[c]{36mm}{\includegraphics[width=32mm]{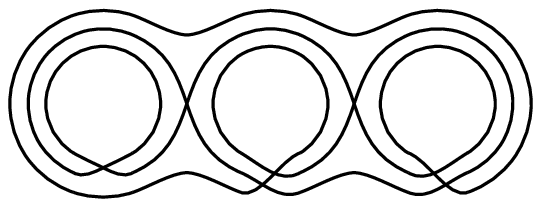}}
$2^1_{9}$\parbox[c]{36mm}{\includegraphics[width=32mm]{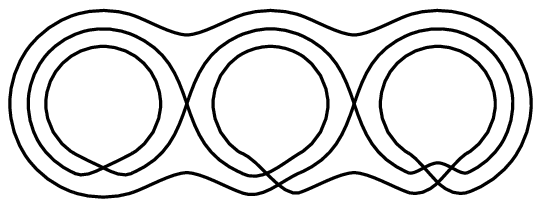}}
$2^1_{10}$\parbox[c]{36mm}{\includegraphics[width=32mm]{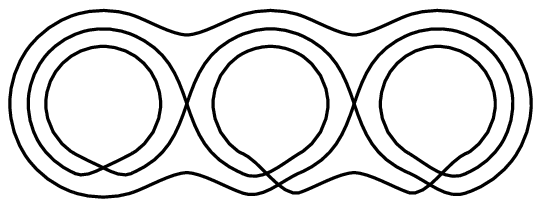}}
$2^1_{11}$\parbox[c]{36mm}{\includegraphics[width=32mm]{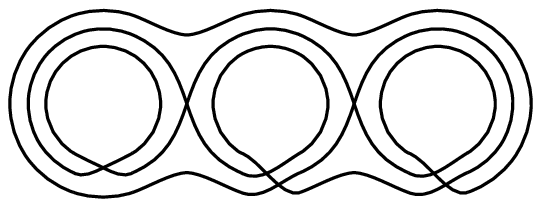}}
$2^1_{12}$\parbox[c]{36mm}{\includegraphics[width=32mm]{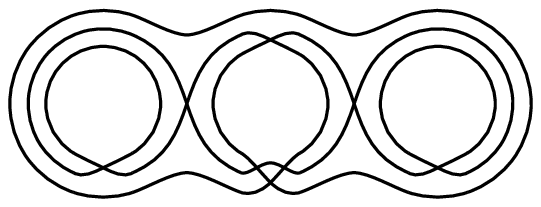}}
$2^1_{13}$\parbox[c]{36mm}{\includegraphics[width=32mm]{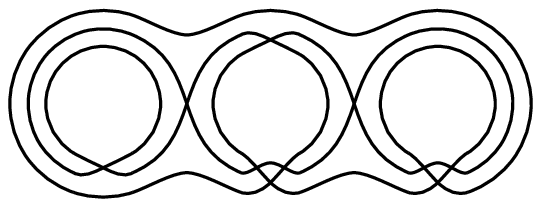}}
$2^1_{14}$\parbox[c]{36mm}{\includegraphics[width=32mm]{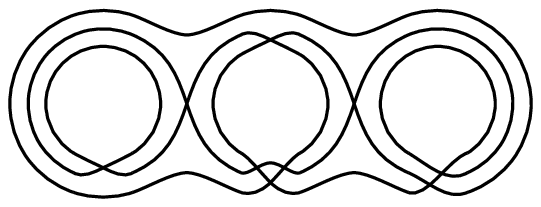}}
$2^1_{15}$\parbox[c]{36mm}{\includegraphics[width=32mm]{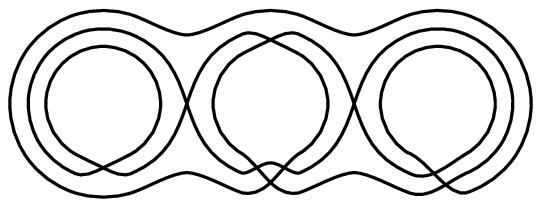}}
$2^1_{16}$\parbox[c]{36mm}{\includegraphics[width=32mm]{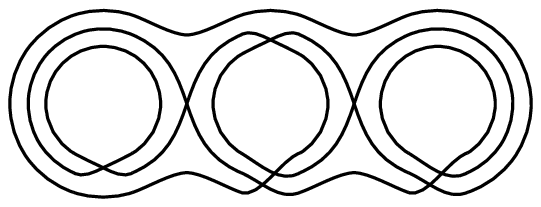}}
$2^1_{17}$\parbox[c]{36mm}{\includegraphics[width=32mm]{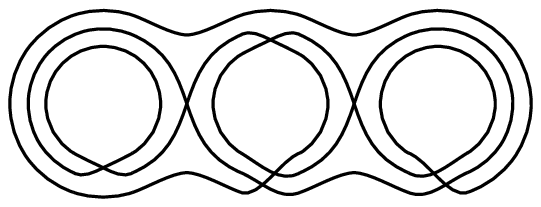}}
$2^1_{18}$\parbox[c]{36mm}{\includegraphics[width=32mm]{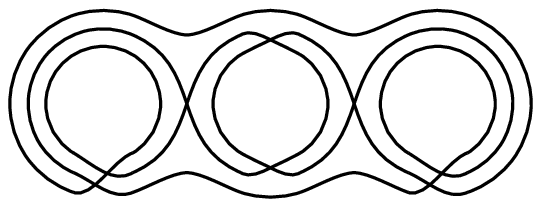}}
$2^1_{19}$\parbox[c]{36mm}{\includegraphics[width=32mm]{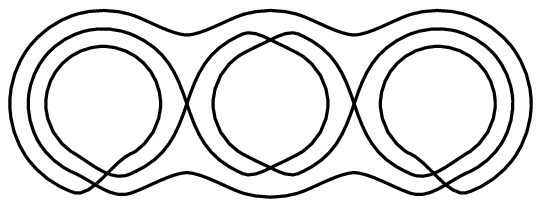}}
$2^1_{20}$\parbox[c]{36mm}{\includegraphics[width=32mm]{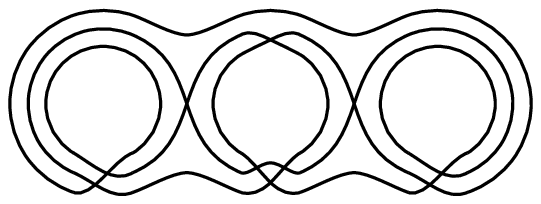}}
$2^1_{21}$\parbox[c]{36mm}{\includegraphics[width=32mm]{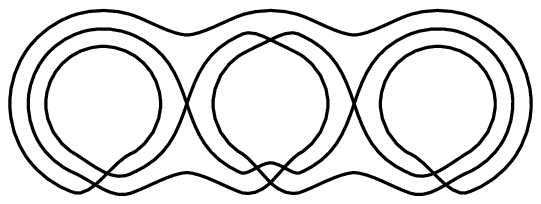}}
$2^1_{22}$\parbox[c]{36mm}{\includegraphics[width=32mm]{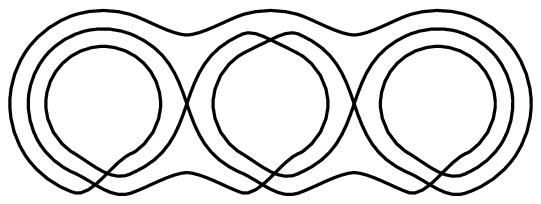}}
$2^1_{23}$\parbox[c]{36mm}{\includegraphics[width=32mm]{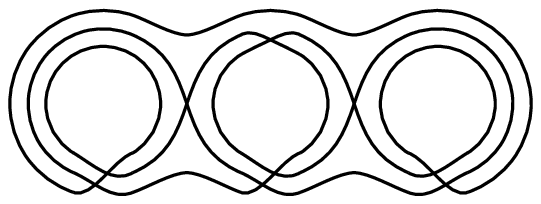}}
$2^1_{24}$\parbox[c]{21mm}{\includegraphics[width=18mm]{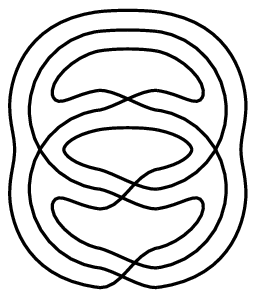}}
$2^1_{25}$\parbox[c]{21mm}{\includegraphics[width=18mm]{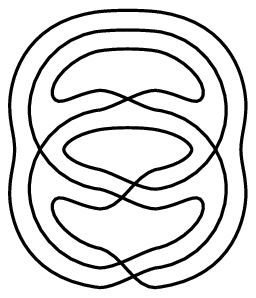}}
$2^1_{26}$\parbox[c]{21mm}{\includegraphics[width=18mm]{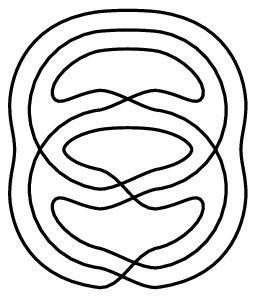}}
$2^1_{27}$\parbox[c]{21mm}{\includegraphics[width=18mm]{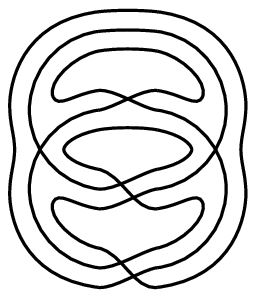}}
$2^1_{28}$\parbox[c]{21mm}{\includegraphics[width=18mm]{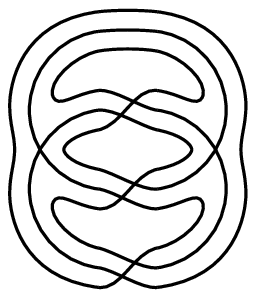}}
$2^1_{29}$\parbox[c]{21mm}{\includegraphics[width=18mm]{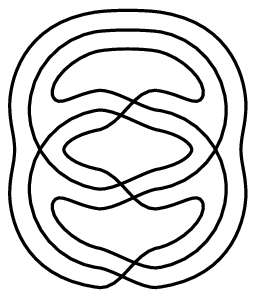}}
$2^1_{30}$\parbox[c]{21mm}{\includegraphics[width=18mm]{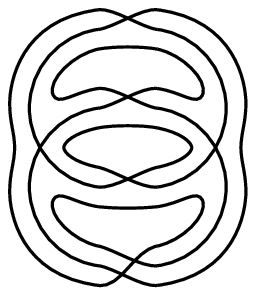}}
$2^1_{31}$\parbox[c]{21mm}{\includegraphics[width=18mm]{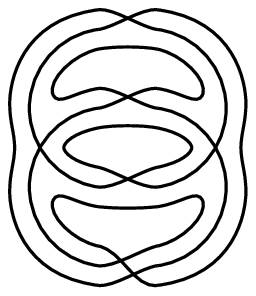}}
$2^1_{32}$\parbox[c]{21mm}{\includegraphics[width=18mm]{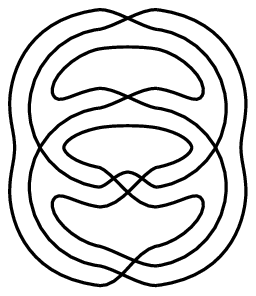}}
$2^1_{33}$\parbox[c]{21mm}{\includegraphics[width=18mm]{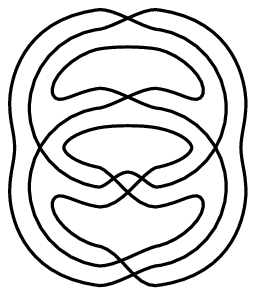}}
$2^1_{34}$\parbox[c]{21mm}{\includegraphics[width=18mm]{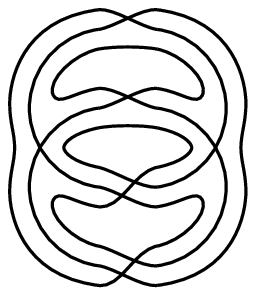}}
$2^1_{35}$\parbox[c]{21mm}{\includegraphics[width=18mm]{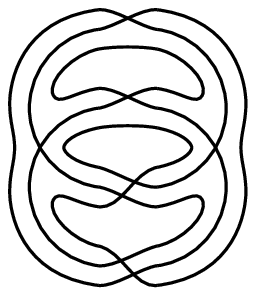}}
$2^1_{36}$\parbox[c]{21mm}{\includegraphics[width=18mm]{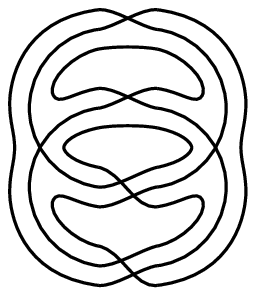}}
$2^1_{37}$\parbox[c]{21mm}{\includegraphics[width=18mm]{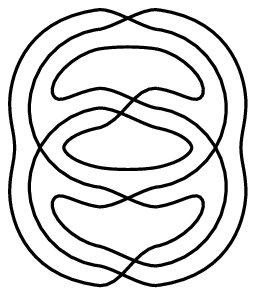}}
$2^1_{38}$\parbox[c]{21mm}{\includegraphics[width=18mm]{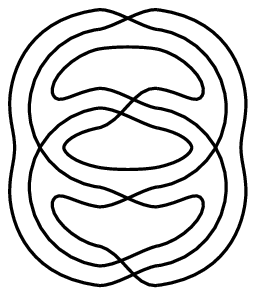}}
$2^1_{39}$\parbox[c]{21mm}{\includegraphics[width=18mm]{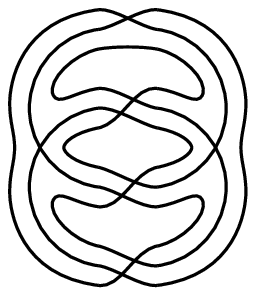}}
$2^1_{40}$\parbox[c]{21mm}{\includegraphics[width=18mm]{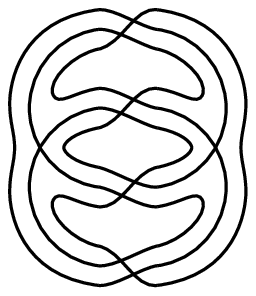}}
\end{center}
\subsection{Special polyhedra with $2$ vertices and $2$ regions}
\begin{center}
$2^2_{1}$\parbox[c]{36mm}{\includegraphics[width=32mm]{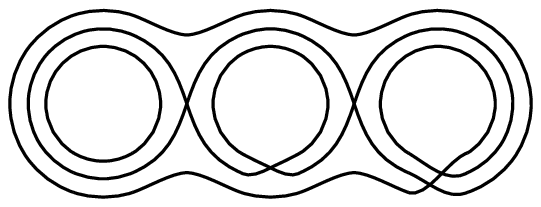}}
$2^2_{2}$\parbox[c]{36mm}{\includegraphics[width=32mm]{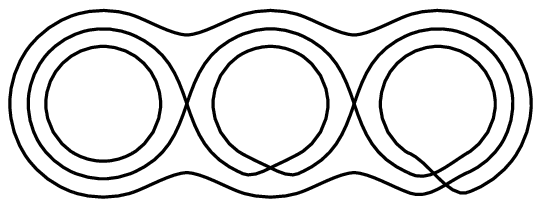}}
$2^2_{3}$\parbox[c]{36mm}{\includegraphics[width=32mm]{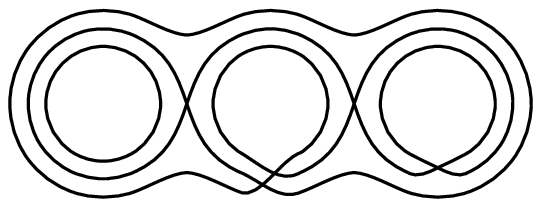}}
$2^2_{4}$\parbox[c]{36mm}{\includegraphics[width=32mm]{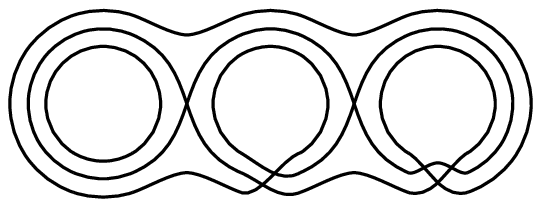}}
$2^2_{5}$\parbox[c]{36mm}{\includegraphics[width=32mm]{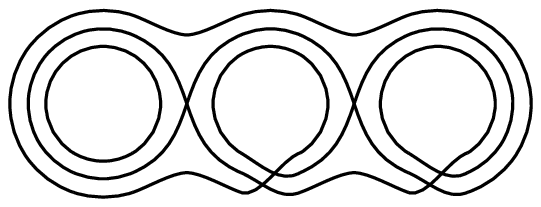}}
$2^2_{6}$\parbox[c]{36mm}{\includegraphics[width=32mm]{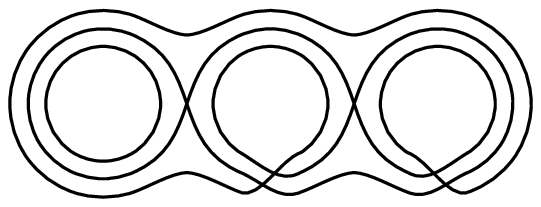}}
$2^2_{7}$\parbox[c]{36mm}{\includegraphics[width=32mm]{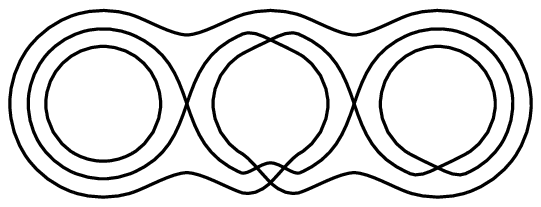}}
$2^2_{8}$\parbox[c]{36mm}{\includegraphics[width=32mm]{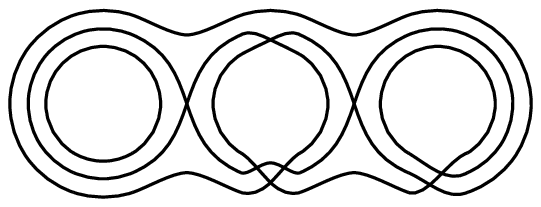}}
$2^2_{9}$\parbox[c]{36mm}{\includegraphics[width=32mm]{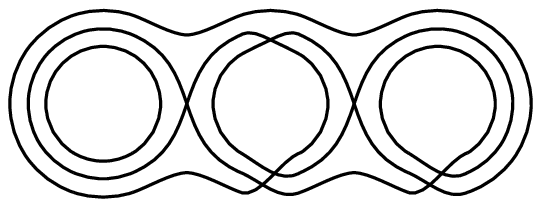}}
$2^2_{10}$\parbox[c]{36mm}{\includegraphics[width=32mm]{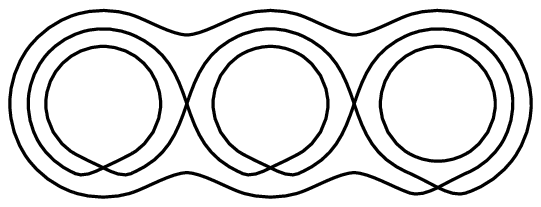}}
$2^2_{11}$\parbox[c]{36mm}{\includegraphics[width=32mm]{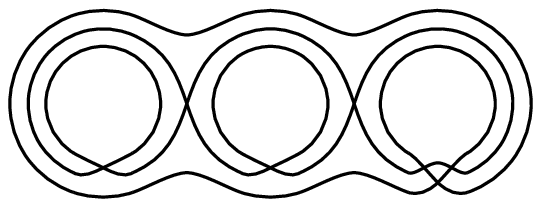}}
$2^2_{12}$\parbox[c]{36mm}{\includegraphics[width=32mm]{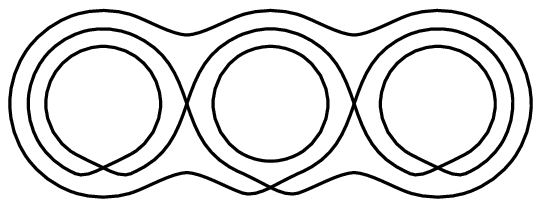}}
$2^2_{13}$\parbox[c]{36mm}{\includegraphics[width=32mm]{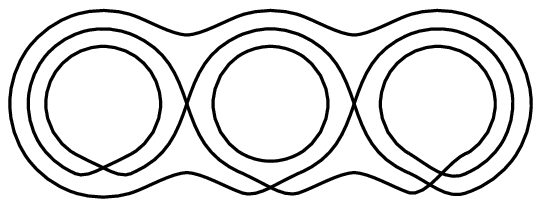}}
$2^2_{14}$\parbox[c]{36mm}{\includegraphics[width=32mm]{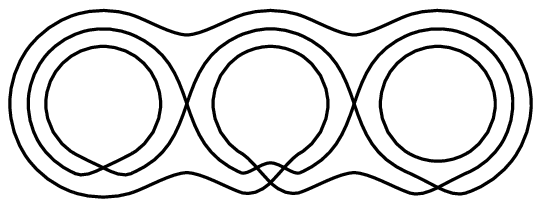}}
$2^2_{15}$\parbox[c]{36mm}{\includegraphics[width=32mm]{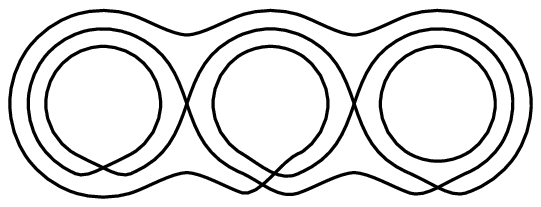}}
$2^2_{16}$\parbox[c]{36mm}{\includegraphics[width=32mm]{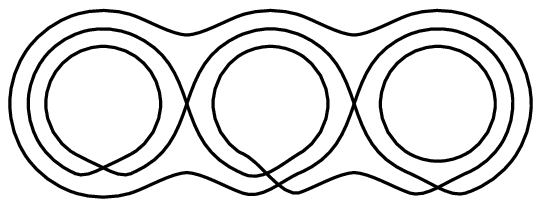}}
$2^2_{17}$\parbox[c]{36mm}{\includegraphics[width=32mm]{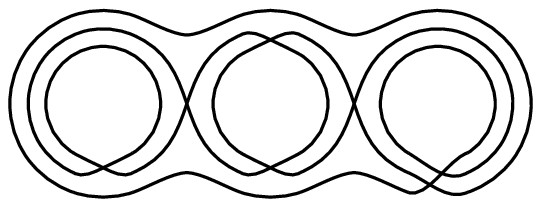}}
$2^2_{18}$\parbox[c]{36mm}{\includegraphics[width=32mm]{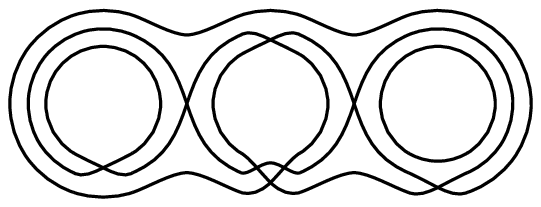}}
$2^2_{19}$\parbox[c]{36mm}{\includegraphics[width=32mm]{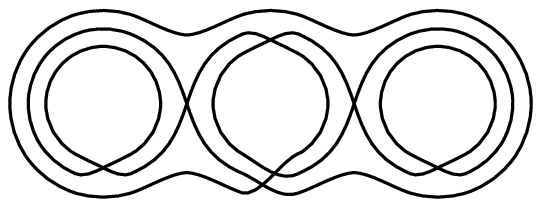}}
$2^2_{20}$\parbox[c]{36mm}{\includegraphics[width=32mm]{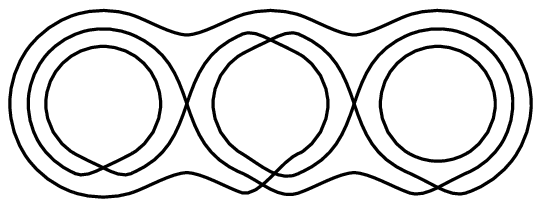}}
$2^2_{21}$\parbox[c]{36mm}{\includegraphics[width=32mm]{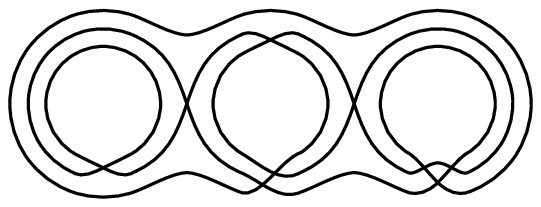}}
$2^2_{22}$\parbox[c]{36mm}{\includegraphics[width=32mm]{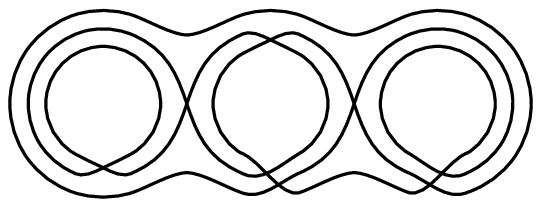}}
$2^2_{23}$\parbox[c]{36mm}{\includegraphics[width=32mm]{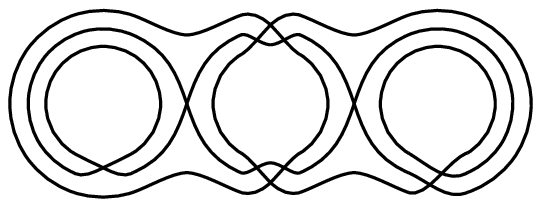}}
$2^2_{24}$\parbox[c]{36mm}{\includegraphics[width=32mm]{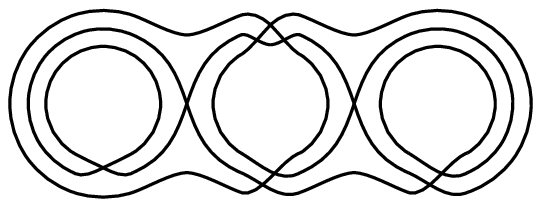}}
$2^2_{25}$\parbox[c]{36mm}{\includegraphics[width=32mm]{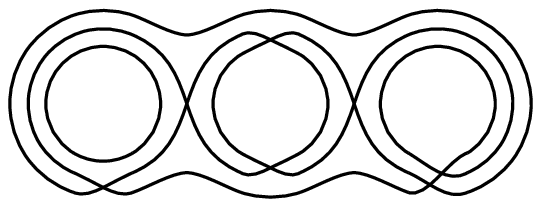}}
$2^2_{26}$\parbox[c]{36mm}{\includegraphics[width=32mm]{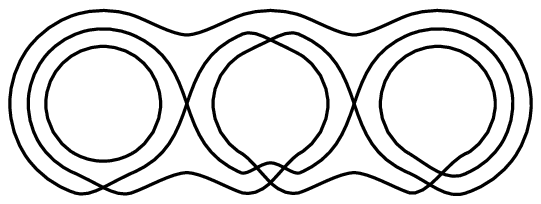}}
$2^2_{27}$\parbox[c]{36mm}{\includegraphics[width=32mm]{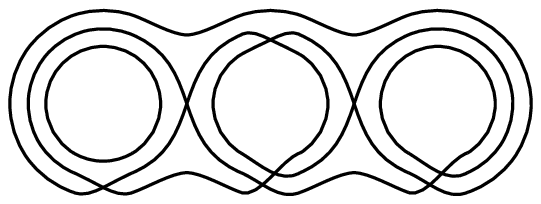}}
$2^2_{28}$\parbox[c]{36mm}{\includegraphics[width=32mm]{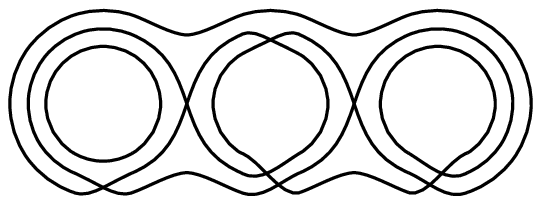}}
$2^2_{29}$\parbox[c]{36mm}{\includegraphics[width=32mm]{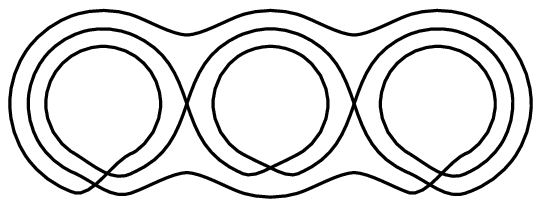}}
$2^2_{30}$\parbox[c]{36mm}{\includegraphics[width=32mm]{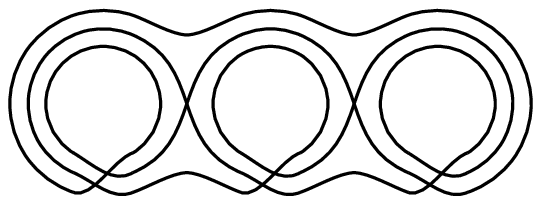}}
$2^2_{31}$\parbox[c]{36mm}{\includegraphics[width=32mm]{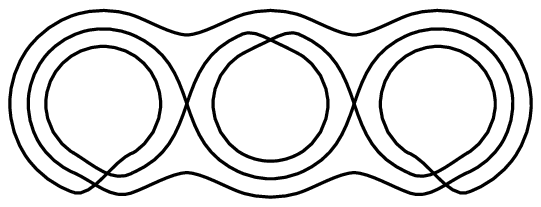}}
$2^2_{32}$\parbox[c]{36mm}{\includegraphics[width=32mm]{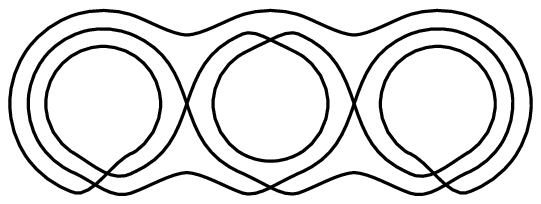}}
$2^2_{33}$\parbox[c]{21mm}{\includegraphics[width=18mm]{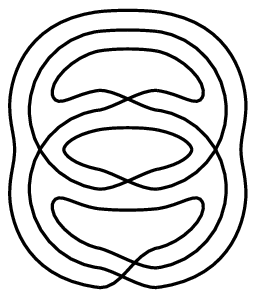}}
$2^2_{34}$\parbox[c]{21mm}{\includegraphics[width=18mm]{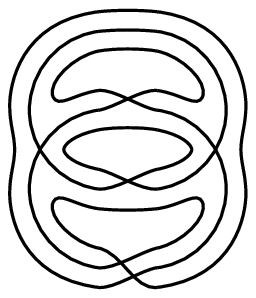}}
$2^2_{35}$\parbox[c]{21mm}{\includegraphics[width=18mm]{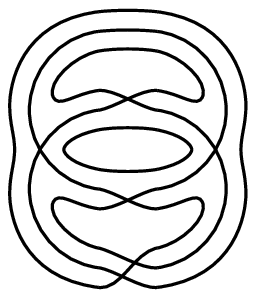}}
$2^2_{36}$\parbox[c]{21mm}{\includegraphics[width=18mm]{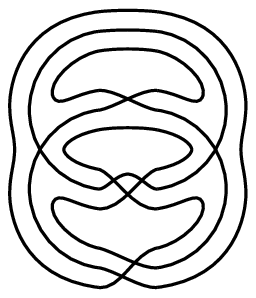}}
$2^2_{37}$\parbox[c]{21mm}{\includegraphics[width=18mm]{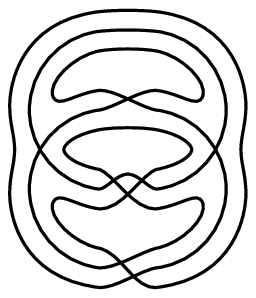}}
$2^2_{38}$\parbox[c]{21mm}{\includegraphics[width=18mm]{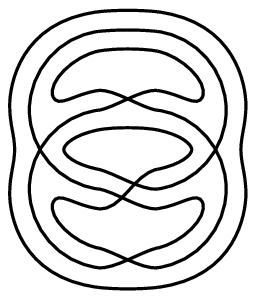}}
$2^2_{39}$\parbox[c]{21mm}{\includegraphics[width=18mm]{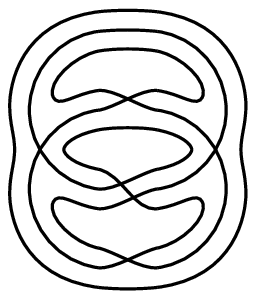}}
$2^2_{40}$\parbox[c]{21mm}{\includegraphics[width=18mm]{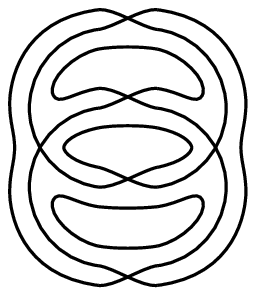}}
$2^2_{41}$\parbox[c]{21mm}{\includegraphics[width=18mm]{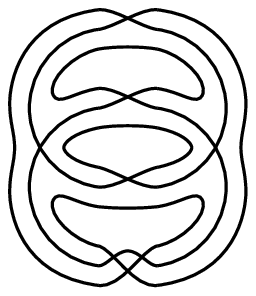}}
$2^2_{42}$\parbox[c]{21mm}{\includegraphics[width=18mm]{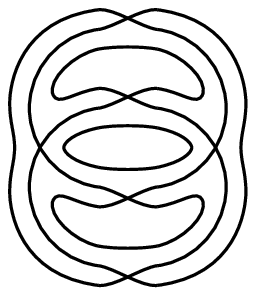}}
$2^2_{43}$\parbox[c]{21mm}{\includegraphics[width=18mm]{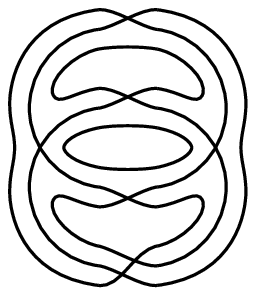}}
$2^2_{44}$\parbox[c]{21mm}{\includegraphics[width=18mm]{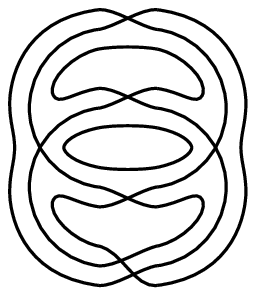}}
$2^2_{45}$\parbox[c]{21mm}{\includegraphics[width=18mm]{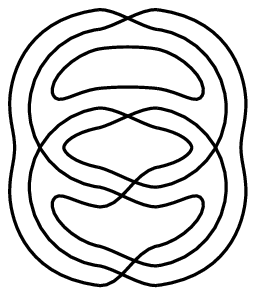}}
$2^2_{46}$\parbox[c]{21mm}{\includegraphics[width=18mm]{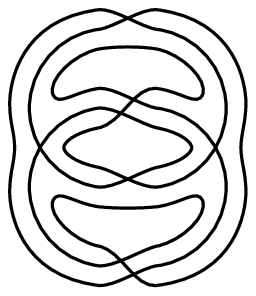}}
$2^2_{47}$\parbox[c]{21mm}{\includegraphics[width=18mm]{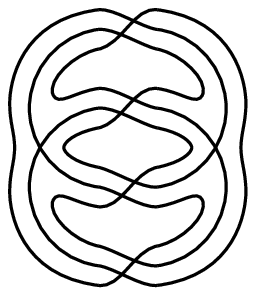}}
$2^2_{48}$\parbox[c]{36mm}{\includegraphics[width=32mm]{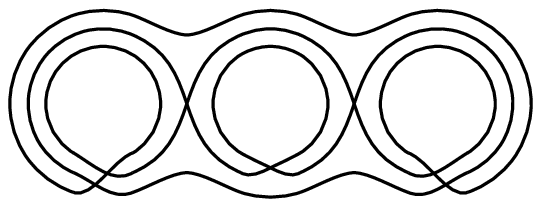}}
$2^2_{49}$\parbox[c]{36mm}{\includegraphics[width=32mm]{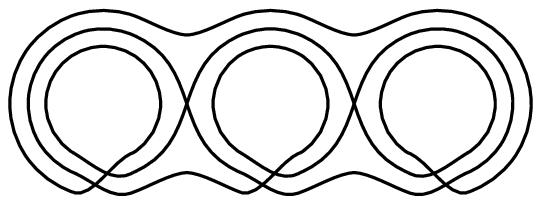}}
$2^2_{50}$\parbox[c]{21mm}{\includegraphics[width=18mm]{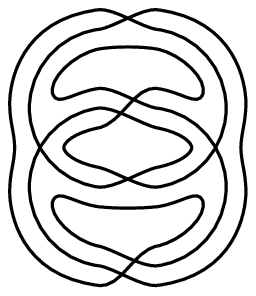}}
$2^2_{51}$\parbox[c]{21mm}{\includegraphics[width=18mm]{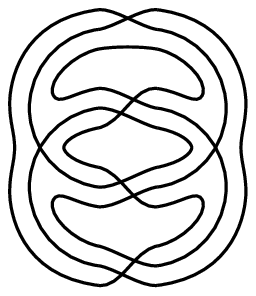}}
$2^2_{52}$\parbox[c]{36mm}{\includegraphics[width=32mm]{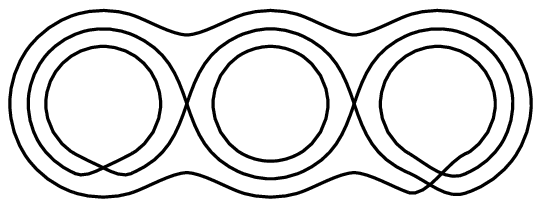}}
$2^2_{53}$\parbox[c]{21mm}{\includegraphics[width=18mm]{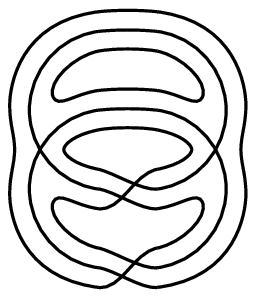}}
$2^2_{54}$\parbox[c]{21mm}{\includegraphics[width=18mm]{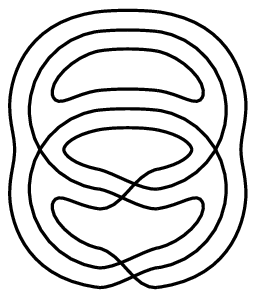}}
$2^2_{55}$\parbox[c]{36mm}{\includegraphics[width=32mm]{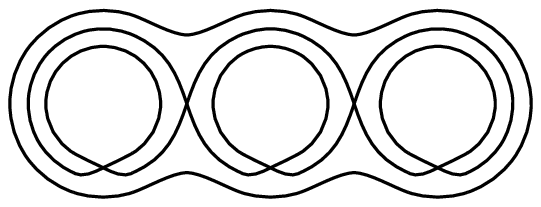}}
$2^2_{56}$\parbox[c]{36mm}{\includegraphics[width=32mm]{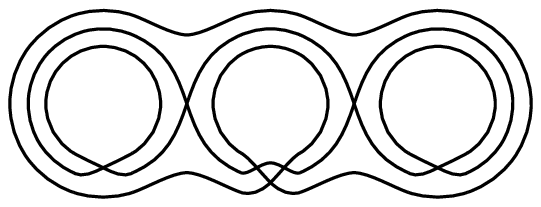}}
$2^2_{57}$\parbox[c]{21mm}{\includegraphics[width=18mm]{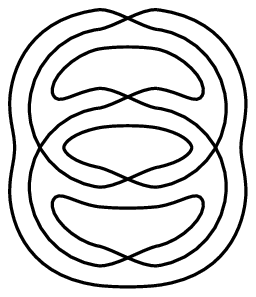}}
$2^2_{58}$\parbox[c]{21mm}{\includegraphics[width=18mm]{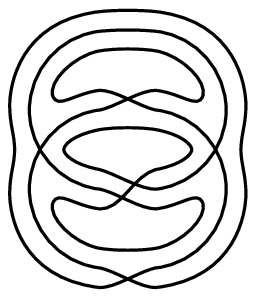}}
\end{center}
\subsection{Special polyhedra with $2$ vertices and $3$ regions}
\label{subsec:Special polyhedra with 2 vertices and 3 regions}
\begin{center}
$2^3_{1}$\parbox[c]{36mm}{\includegraphics[width=32mm]{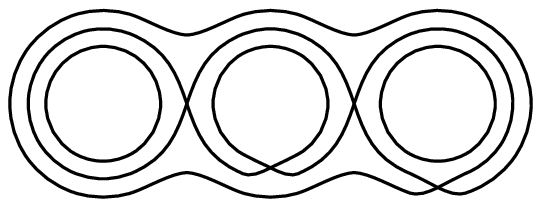}}
$2^3_{2}$\parbox[c]{36mm}{\includegraphics[width=32mm]{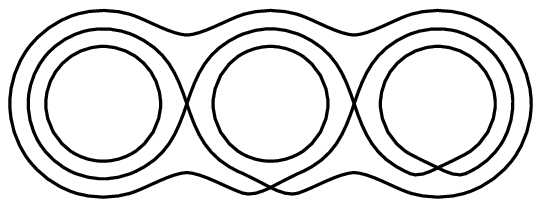}}
$2^3_{3}$\parbox[c]{36mm}{\includegraphics[width=32mm]{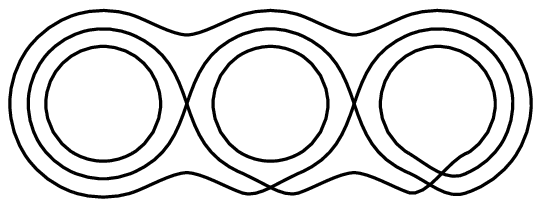}}
$2^3_{4}$\parbox[c]{36mm}{\includegraphics[width=32mm]{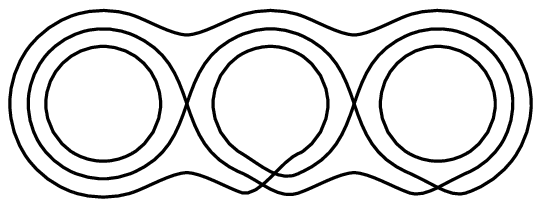}}
$2^3_{5}$\parbox[c]{36mm}{\includegraphics[width=32mm]{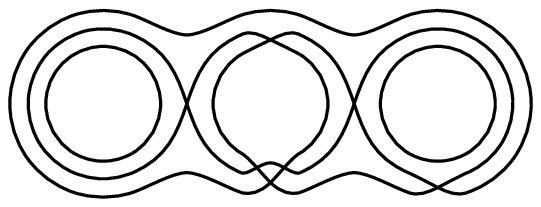}}
$2^3_{6}$\parbox[c]{36mm}{\includegraphics[width=32mm]{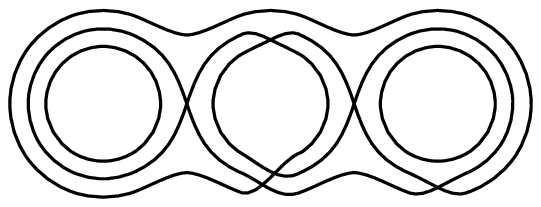}}
$2^3_{7}$\parbox[c]{36mm}{\includegraphics[width=32mm]{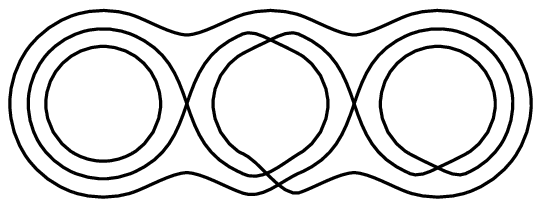}}
$2^3_{8}$\parbox[c]{36mm}{\includegraphics[width=32mm]{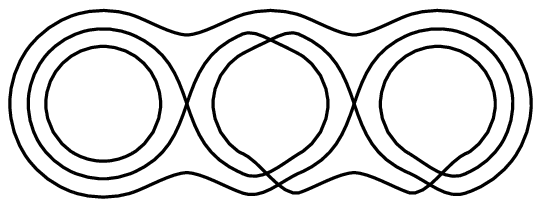}}
$2^3_{9}$\parbox[c]{36mm}{\includegraphics[width=32mm]{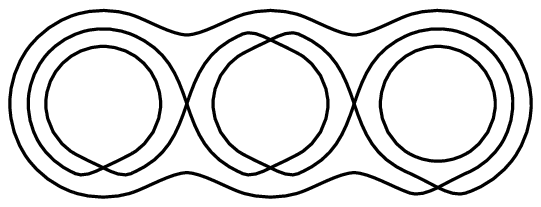}}
$2^3_{10}$\parbox[c]{36mm}{\includegraphics[width=32mm]{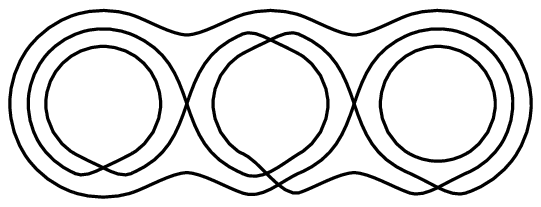}}
$2^3_{11}$\parbox[c]{36mm}{\includegraphics[width=32mm]{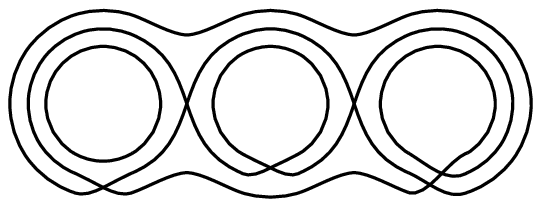}}
$2^3_{12}$\parbox[c]{36mm}{\includegraphics[width=32mm]{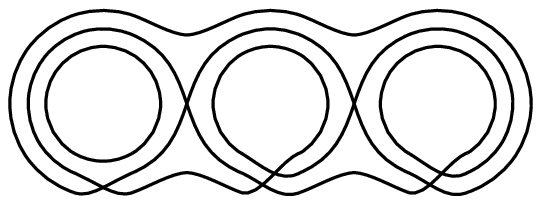}}
$2^3_{13}$\parbox[c]{36mm}{\includegraphics[width=32mm]{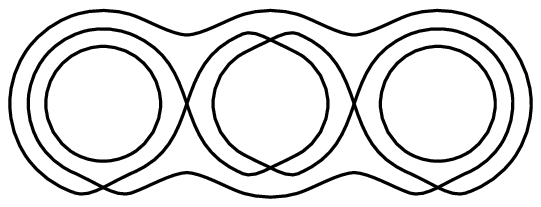}}
$2^3_{14}$\parbox[c]{36mm}{\includegraphics[width=32mm]{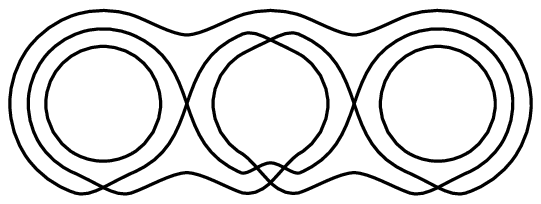}}
$2^3_{15}$\parbox[c]{36mm}{\includegraphics[width=32mm]{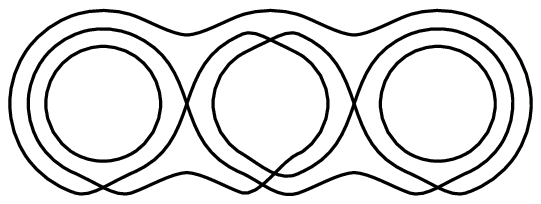}}
$2^3_{16}$\parbox[c]{21mm}{\includegraphics[width=18mm]{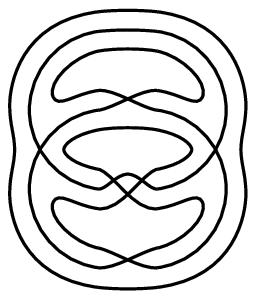}}
$2^3_{17}$\parbox[c]{36mm}{\includegraphics[width=32mm]{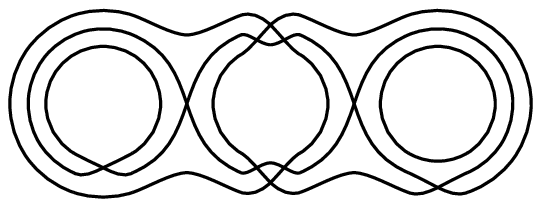}}
$2^3_{18}$\parbox[c]{36mm}{\includegraphics[width=32mm]{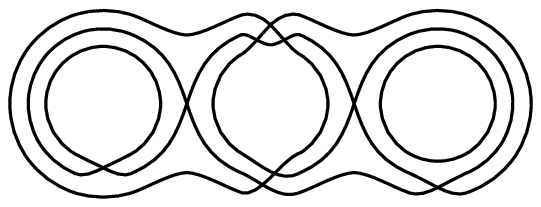}}
$2^3_{19}$\parbox[c]{36mm}{\includegraphics[width=32mm]{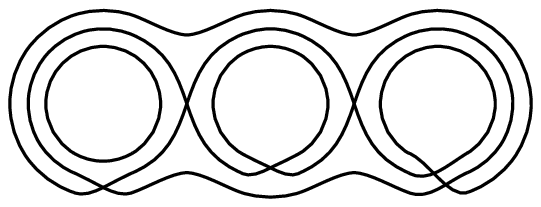}}
$2^3_{20}$\parbox[c]{36mm}{\includegraphics[width=32mm]{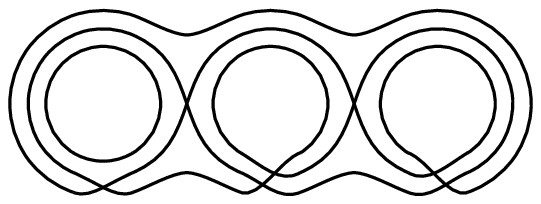}}
$2^3_{21}$\parbox[c]{21mm}{\includegraphics[width=18mm]{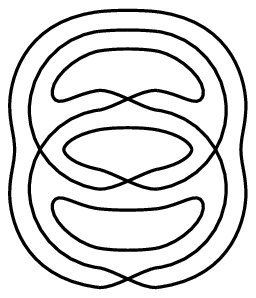}}
$2^3_{22}$\parbox[c]{21mm}{\includegraphics[width=18mm]{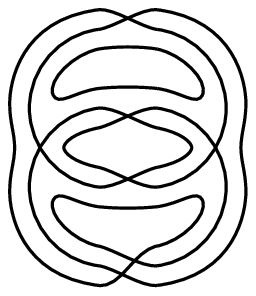}}
$2^3_{23}$\parbox[c]{36mm}{\includegraphics[width=32mm]{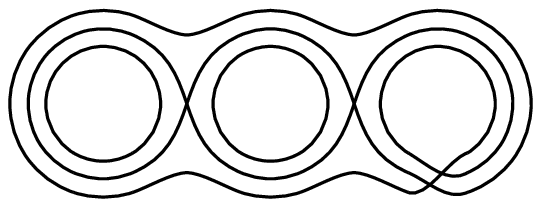}}
$2^3_{24}$\parbox[c]{36mm}{\includegraphics[width=32mm]{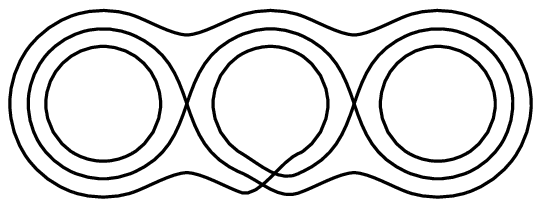}}
$2^3_{25}$\parbox[c]{36mm}{\includegraphics[width=32mm]{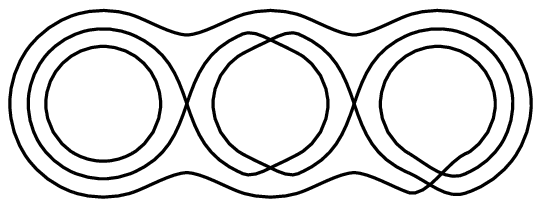}}
$2^3_{26}$\parbox[c]{36mm}{\includegraphics[width=32mm]{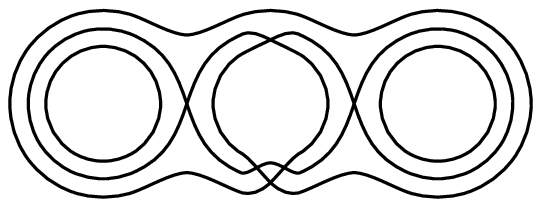}}
$2^3_{27}$\parbox[c]{36mm}{\includegraphics[width=32mm]{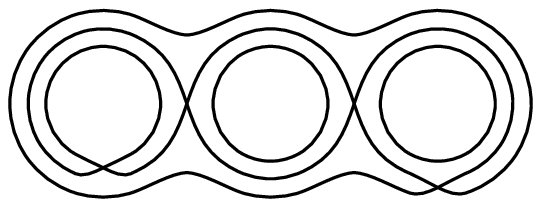}}
$2^3_{28}$\parbox[c]{36mm}{\includegraphics[width=32mm]{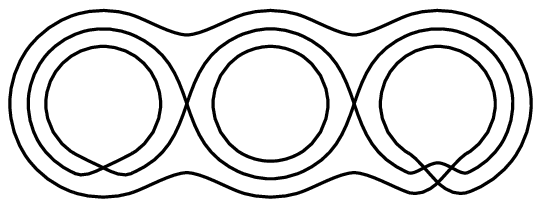}}
$2^3_{29}$\parbox[c]{36mm}{\includegraphics[width=32mm]{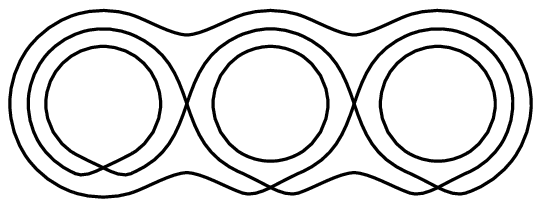}}
$2^3_{30}$\parbox[c]{36mm}{\includegraphics[width=32mm]{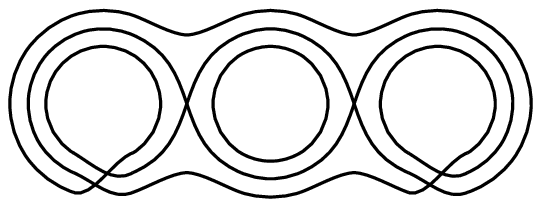}}
$2^3_{31}$\parbox[c]{21mm}{\includegraphics[width=18mm]{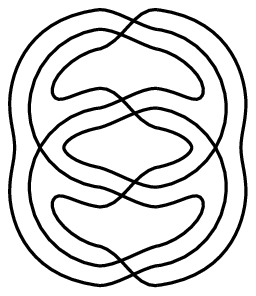}}
$2^3_{32}$\parbox[c]{36mm}{\includegraphics[width=32mm]{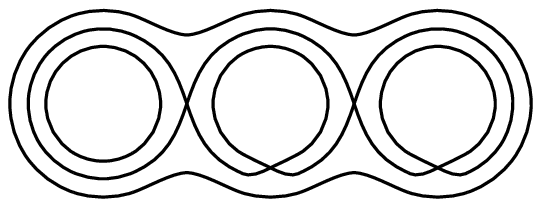}}
$2^3_{33}$\parbox[c]{36mm}{\includegraphics[width=32mm]{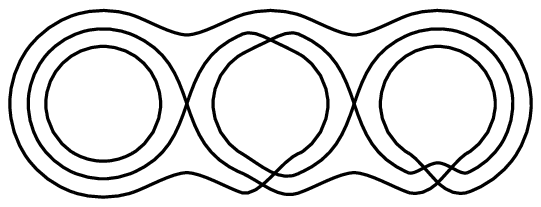}}
$2^3_{34}$\parbox[c]{36mm}{\includegraphics[width=32mm]{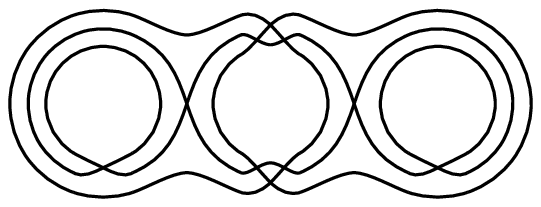}}
$2^3_{35}$\parbox[c]{36mm}{\includegraphics[width=32mm]{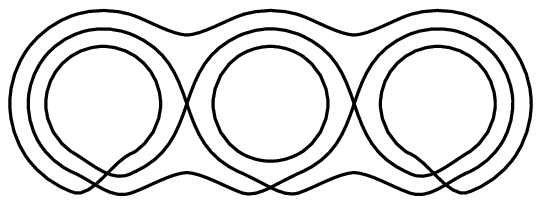}}
$2^3_{36}$\parbox[c]{21mm}{\includegraphics[width=18mm]{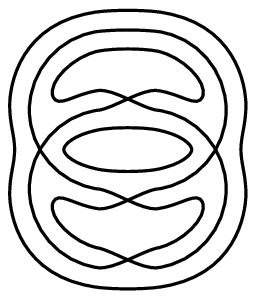}}
$2^3_{37}$\parbox[c]{21mm}{\includegraphics[width=18mm]{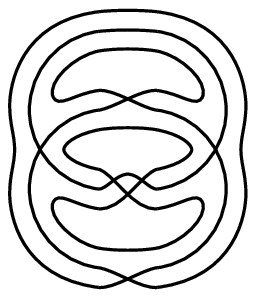}}
$2^3_{38}$\parbox[c]{36mm}{\includegraphics[width=32mm]{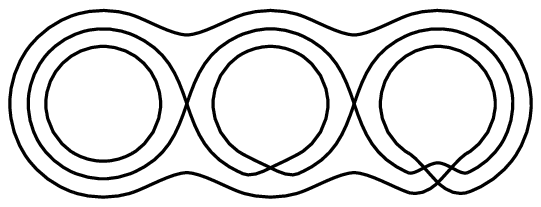}}
$2^3_{39}$\parbox[c]{36mm}{\includegraphics[width=32mm]{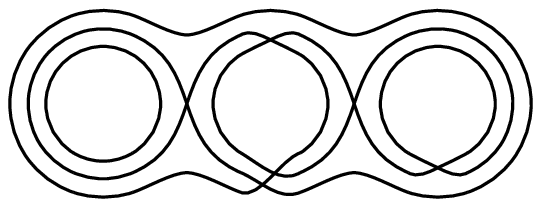}}
$2^3_{40}$\parbox[c]{36mm}{\includegraphics[width=32mm]{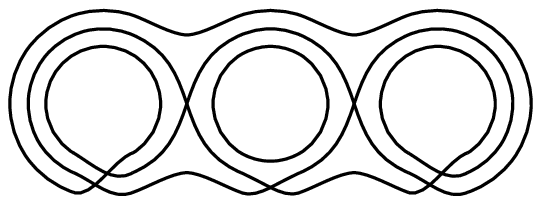}}
$2^3_{41}$\parbox[c]{21mm}{\includegraphics[width=18mm]{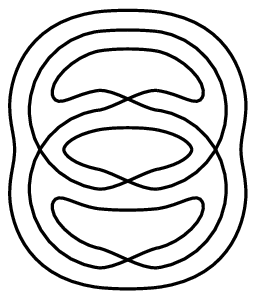}}
$2^3_{42}$\parbox[c]{21mm}{\includegraphics[width=18mm]{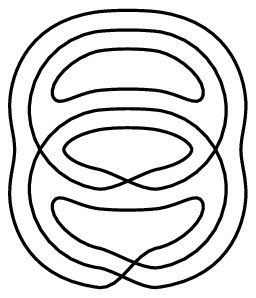}}
$2^3_{43}$\parbox[c]{36mm}{\includegraphics[width=32mm]{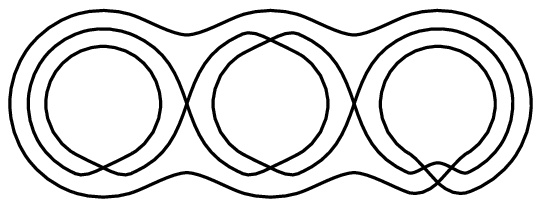}}
$2^3_{44}$\parbox[c]{21mm}{\includegraphics[width=18mm]{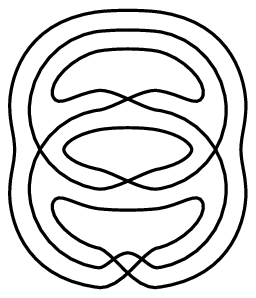}}
$2^3_{45}$\parbox[c]{36mm}{\includegraphics[width=32mm]{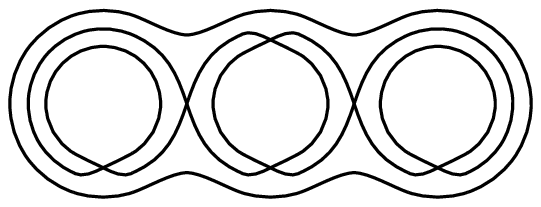}}
$2^3_{46}$\parbox[c]{36mm}{\includegraphics[width=32mm]{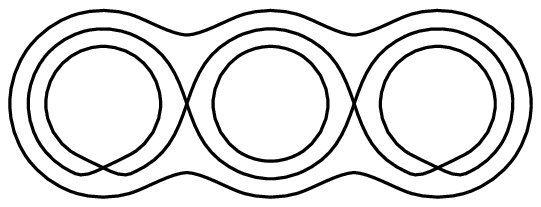}}
$2^3_{47}$\parbox[c]{36mm}{\includegraphics[width=32mm]{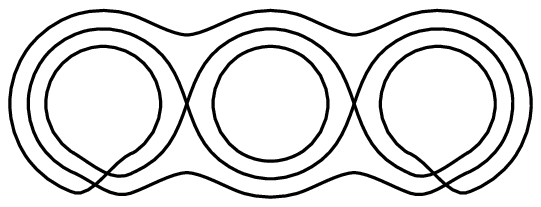}}
$2^3_{48}$\parbox[c]{21mm}{\includegraphics[width=18mm]{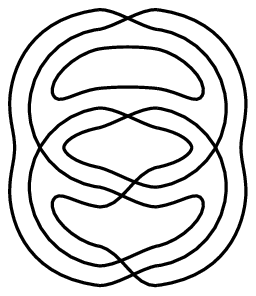}}
\end{center}
\subsection{Special polyhedra with $2$ vertices and $4$ regions}
\begin{center}
$2^4_{1}$\parbox[c]{36mm}{\includegraphics[width=32mm]{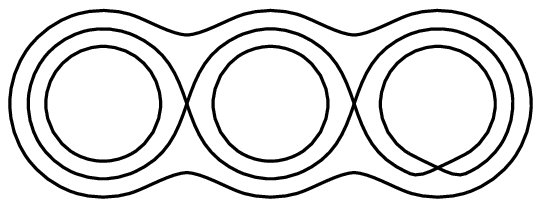}}
$2^4_{2}$\parbox[c]{36mm}{\includegraphics[width=32mm]{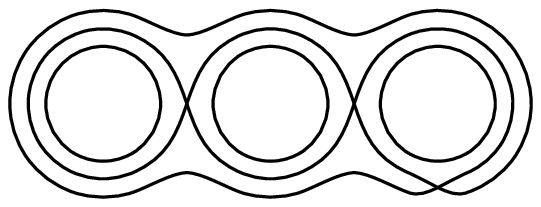}}
$2^4_{3}$\parbox[c]{36mm}{\includegraphics[width=32mm]{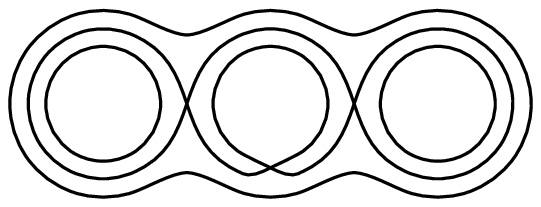}}
$2^4_{4}$\parbox[c]{36mm}{\includegraphics[width=32mm]{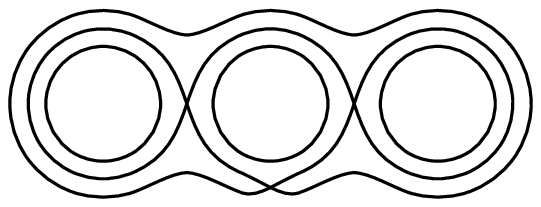}}
$2^4_{5}$\parbox[c]{36mm}{\includegraphics[width=32mm]{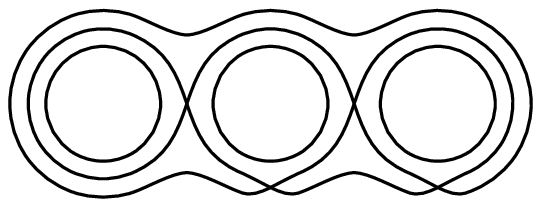}}
$2^4_{6}$\parbox[c]{36mm}{\includegraphics[width=32mm]{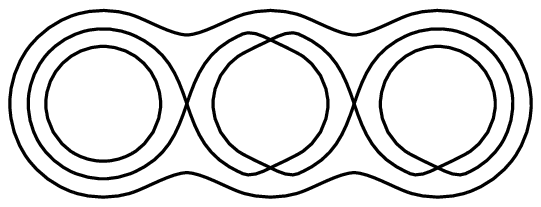}}
$2^4_{7}$\parbox[c]{36mm}{\includegraphics[width=32mm]{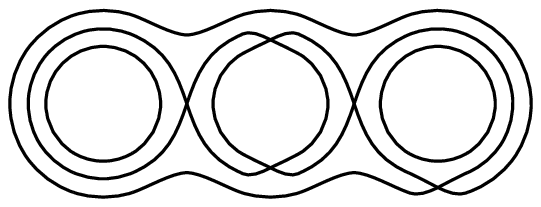}}
$2^4_{8}$\parbox[c]{36mm}{\includegraphics[width=32mm]{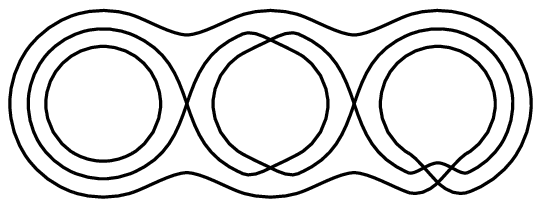}}
$2^4_{9}$\parbox[c]{36mm}{\includegraphics[width=32mm]{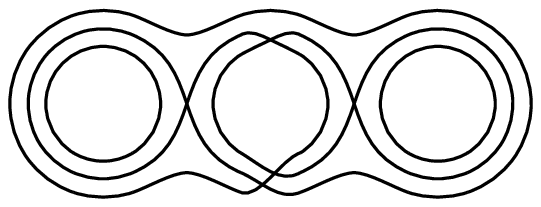}}
$2^4_{10}$\parbox[c]{36mm}{\includegraphics[width=32mm]{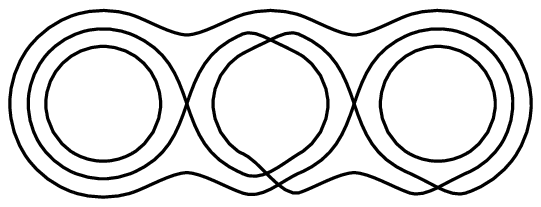}}
$2^4_{11}$\parbox[c]{36mm}{\includegraphics[width=32mm]{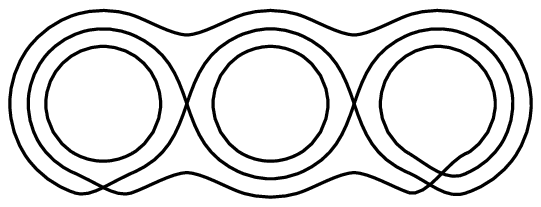}}
$2^4_{12}$\parbox[c]{36mm}{\includegraphics[width=32mm]{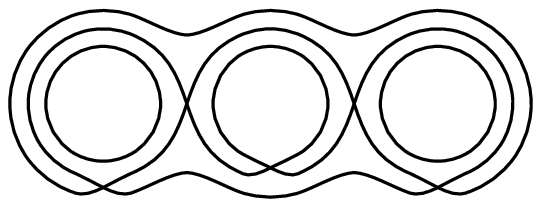}}
$2^4_{13}$\parbox[c]{36mm}{\includegraphics[width=32mm]{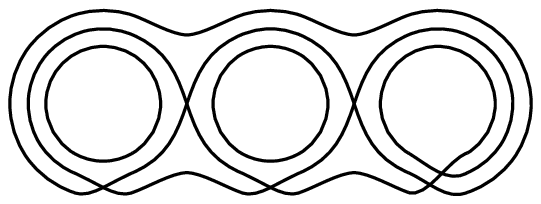}}
$2^4_{14}$\parbox[c]{36mm}{\includegraphics[width=32mm]{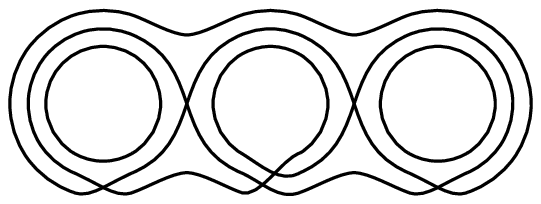}}
$2^4_{15}$\parbox[c]{21mm}{\includegraphics[width=18mm]{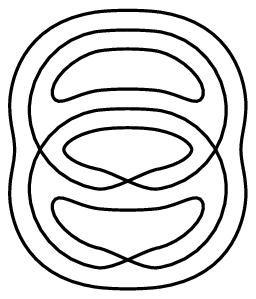}}
$2^4_{16}$\parbox[c]{21mm}{\includegraphics[width=18mm]{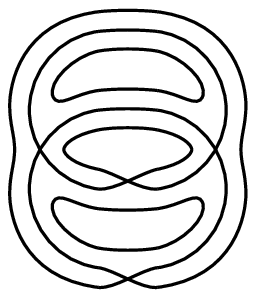}}
$2^4_{17}$\parbox[c]{21mm}{\includegraphics[width=18mm]{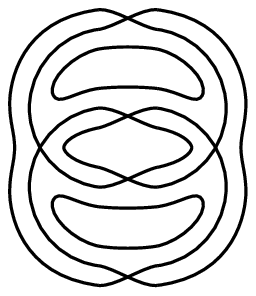}}
$2^4_{18}$\parbox[c]{21mm}{\includegraphics[width=18mm]{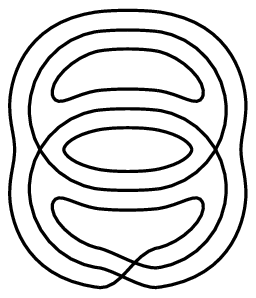}}
\end{center}
\subsection{Special polyhedra with $2$ vertices and $5$ regions}
\begin{center}
$2^5_{1}$\parbox[c]{36mm}{\includegraphics[width=32mm]{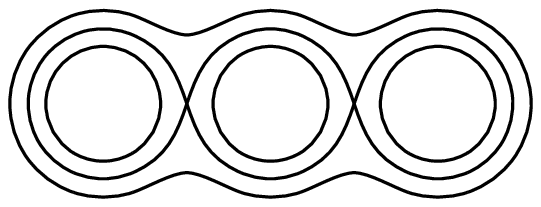}}
$2^5_{2}$\parbox[c]{36mm}{\includegraphics[width=32mm]{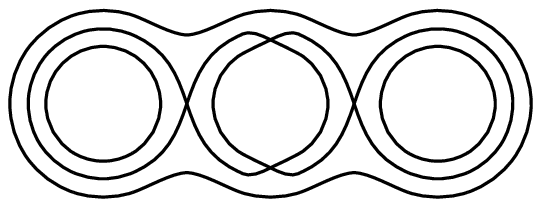}}
$2^5_{3}$\parbox[c]{36mm}{\includegraphics[width=32mm]{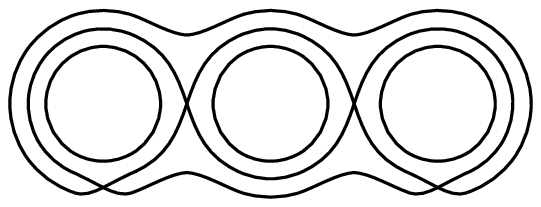}}
$2^5_{4}$\parbox[c]{36mm}{\includegraphics[width=32mm]{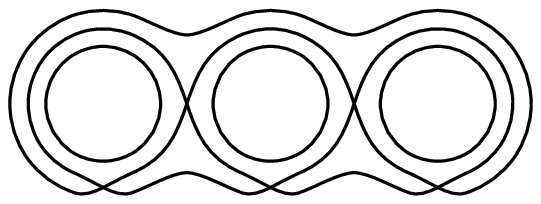}}
$2^5_{5}$\parbox[c]{21mm}{\includegraphics[width=18mm]{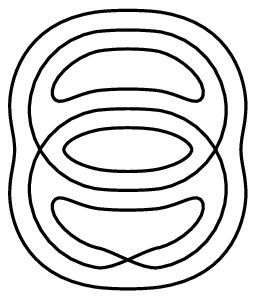}}
\end{center}
\subsection{Special polyhedra with $2$ vertices and $6$ regions}
\begin{center}
$2^6_1$\parbox[c]{21mm}{\includegraphics[width=18mm]{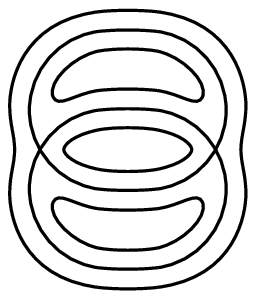}}
\end{center}

\end{document}